\documentclass[a4paper,11pt,notitlepage,reqno]{amsart}

\usepackage[utf8]{inputenc}
\usepackage[T1]{fontenc}
\usepackage{lmodern}
\usepackage[provide=*,main=english,spanish]{babel}
\usepackage{csquotes}

\usepackage[
    margin=28mm,
    marginparwidth=25mm,
    marginparsep=2mm,
    bottom=25mm
]{geometry}
\usepackage[shortlabels]{enumitem}
\usepackage{flafter}
\usepackage{graphicx}

\usepackage{xcolor}
\definecolor{linkInternal}{HTML}{6A1B9A}
\definecolor{linkCitation}{HTML}{00796B}
\definecolor{linkExternal}{HTML}{1565C0}

\usepackage{amssymb}
\usepackage{mathrsfs}
\usepackage{mathtools}

\usepackage[
    backend=biber,
    style=numeric,
    sorting=nyt,
    maxbibnames=10,
    minbibnames=1,
    url=false,
    doi=true,
    eprint=true
]{biblatex}

\DeclareFieldFormat*{title}{\mkbibemph{#1}}
\DeclareFieldFormat*{journaltitle}{#1}
\DeclareFieldFormat*{booktitle}{#1}
\DeclareFieldFormat[article]{volume}{\mkbibbold{#1}}
\DeclareFieldFormat[article]{number}{\bibstring{number}~#1}
\DeclareFieldFormat[article]{pages}{#1}

\renewbibmacro{in:}{%
  \ifentrytype{article}{}{\printtext{\bibstring{in}\intitlepunct}}}

\renewbibmacro*{journal+issuetitle}{%
  \usebibmacro{journal}%
  \setunit*{\addspace}%
  \iffieldundef{series}{}{\newunit\printfield{series}\setunit{\addspace}}%
  \printfield{volume}%
  \setunit{\addspace}%
  \usebibmacro{issue+date}%
  \setunit{\addcolon\space}%
  \usebibmacro{issue}%
  \setunit{\addcomma\space}%
  \printfield{number}%
  \setunit{\addcomma\space}%
  \printfield{eid}%
  \newunit}

\usepackage[
    pdfencoding=unicode,
    psdextra,
    colorlinks=true,
    linkcolor=linkInternal,
    citecolor=linkCitation,
    urlcolor=linkExternal,
    hypertexnames=false
]{hyperref}

\usepackage{cleveref}

\DeclarePairedDelimiter{\abs}{\lvert}{\rvert}
\DeclarePairedDelimiter{\norm}{\lVert}{\rVert}

\newcommand{\quotient}[2]{{\raisebox{.2em}{$#1$}\left/\raisebox{-.2em}{$#2$}\right.}}

\DeclareMathOperator{\ran}{Ran}
\DeclareMathOperator{\lat}{Lat}
\DeclareMathOperator{\dist}{dist}

\DeclareMathOperator{\codim}{codim}

\newcommand{\Real}{\mathfrak{Re}\,}

\newcommand{\N}{\mathbb{N}}
\newcommand{\Z}{\mathbb{Z}}
\newcommand{\R}{\mathbb{R}}
\newcommand{\C}{\mathbb{C}}

\renewcommand{\S}{\mathbb{S}}

\newcommand{\CC}{\mathcal{C}}
\newcommand{\LL}{\mathcal{L}}

\newcommand{\BB}{\mathcal{B}}
\newcommand{\MM}{\mathcal{M}}
\newcommand{\Ca}{\mathcal{C}_\alpha}
\newcommand{\e}{\varepsilon}

\theoremstyle{plain}
\newtheorem{theorem}{Theorem}[section]
\newtheorem*{theorem*}{Theorem}
\newtheorem{corollary}[theorem]{Corollary}
\newtheorem{proposition}[theorem]{Proposition}
\newtheorem{lemma}[theorem]{Lemma}

\theoremstyle{definition}
\newtheorem{definition}[theorem]{Definition}

\theoremstyle{remark}
\newtheorem{remark}[theorem]{Remark}
\newtheorem{example}[theorem]{Example}

\numberwithin{equation}{section}

\makeatletter
\renewcommand*{\subsection}{%
  \@startsection{subsection}{2}%
    {\z@}%
    {.5\linespacing\@plus.7\linespacing}%
    {.5\linespacing}%
    {\normalfont\itshape}}
\makeatother

\begin{document}
	
	\title[Cesàro-Hardy operators]{Cesàro-Hardy operators on $L^p[0,1]$: fine spectrum, weighted Koopman semigroups and invariant subspaces}

	\author[L. Abadias]{Luciano Abadias}
    \address{\newline
        Luciano Abadías \newline
        {\scriptsize Departamento de Matemáticas,
        Instituto Universitario de Matemáticas y Aplicaciones,
        Universidad de Zaragoza,
        50009 Zaragoza, Spain.}}
    \email{labadias@unizar.es}

    \author[A. Mahillo]{Alejandro Mahillo}
    \address{\newline
        Alejandro Mahillo \newline
        Departamento de Matemáticas, Estadística y Computación,
        Universidad de Cantabria,
        39005 Santander, Spain.}
    \email{alejandro.mahillo@unican.es}

    \author[P.J. Miana]{Pedro J. Miana}
    \address{\newline
        Pedro J. Miana \newline
        {\scriptsize Departamento de Matemáticas,
        Instituto Universitario de Matemáticas y Aplicaciones,
        Universidad de Zaragoza,
        50009 Zaragoza, Spain.}}
    \email{pjmiana@unizar.es}

	\thanks{Authors have been partially supported by Grants PID2022-137294NB-I00 and PID2025-169474NB-C22 funded by MICIU/AEI/10.13039/501100011033 and ``ERDF/EU'' and Projects E48-20R and E48-26R, Gobierno de Aragón, Spain.
    The second author has been partially supported by Grant PRE-2020-092311 funded by MICIU/AEI/10.13039/501100011033 and by ``ESF Investing in your future.''}

	\thanks{This is the accepted version of the following article: L.~Abadías, A.~Mahillo and P.~J.~Miana,
    ``Cesàro--Hardy operators on $L^p[0,1]$: Fine spectrum, weighted Koopman semigroups, and invariant subspaces,''
    \textit{J. London Math. Soc.} (2) \textbf{114} (2026), no.~6, e70779, which has been published in final form at
    \mbox{\url{https://doi.org/10.1112/jlms.70779}}.}

	

\begin{abstract}

    In this paper we study boundedness and detailed spectral properties for the Cesàro-Hardy operator and some generalizations in $L^p[0,1]$. The study employs $C_0$-semigroup theory, expressing the Cesàro-Hardy operators and their dual operators through subordination with $C_0$-semigroups $T(t)$ and $S(t)$ respectively. The spectral properties of the semigroup’s infinitesimal generators are transferred to the Cesàro-Hardy operators using functional calculus methods. Furthermore, some implications for the Invariant Subspace Problem are explored by demonstrating the universality of certain translations related to the semigroup $T(t),$  and providing results on the invariant subspaces of these operators.
\end{abstract}

\maketitle

\section{Introduction}
\label{sec:introduction}

Our starting point is the work of A. Brown, P. R. Halmos and A. L. Shields, \cite{BrownHalmosShields1965}. In their paper, the Cesàro-Hardy operators, $C_0, C_1$ and $C_\infty$ were defined on $\ell^2, L^2[0,1]$ and $L^2(0,\infty)$ respectively by
\[
    C_0 f(n) = \frac{1}{n+1} \sum_{j=0}^n f(j),\quad C_1 f(x) = \frac{1}{x} \int_0^x f(s) \, ds \quad \text{and} \quad C_\infty f(x) = \frac{1}{x} \int_0^x f(s) \, ds.
\]
These operators are bounded due to the Hardy inequality (see \cite[Section 9.8]{HardyLittlewoodPolya1952}). Furthermore, the authors characterized their spectrum and point spectrum, offering insights beyond simple boundedness. Building on this foundation, D. W. Boyd \cite{Boyd1968} extended these spectral results for the operator $C_\infty$ to the spaces $L^p(0,\infty)$ for $1 < p \leq \infty$, while G. Leibowitz followed suit with the operators $C_0$ and $C_1$ on the spaces $\ell^p$ and $L^p[0,1]$ for $1 < p \leq \infty$ respectively \cite{Leibowitz72,Leibowitz73}. In \cite{Muntean1980}, the author computed the approximate and continuous spectrum of $C_1,$ and in \cite{Erovenko1987} the essential spectrum of $C_0,C_1$ and $C_{\infty}$ was studied.

Since then, the spectrum and point spectrum of the Cesàro-Hardy operators have been extensively studied across various function spaces, including $C[0,1]$, $C_l[0,\infty]$, $C(0,\infty)$, $L^p_{\text{loc}}(0,\infty)$ with $1 < p < \infty$ \cite{AlbaneseBonetRicker2015}, $C^\infty(0,\infty)$ \cite{AlbaneseBonetRicker2016,AlbaneseBonetRicker2016_Erratum}, and ultra-differentiable function spaces \cite{Albanese2021}. Recent studies have also examined $C_1$ in Lorentz spaces $L_{p,q}(0,1)$ for $1 < p < \infty$ and $1 \leq q \leq \infty$ \cite{TulenovZaur2023} and in rearrangement invariant spaces over a finite interval and a half line \cite{AkhymbekTulenovZaur2024}.

The discrete Cesàro-Hardy operator's spectrum and fine spectra have similarly been analyzed in various spaces, including $\ell_p$ $(1 < p < \infty)$ \cite{Gonzalez1985, CurberaRicker2013}, $c_0$ \cite{Reade1985, AkhmedovBasar2004}, $\ell_\infty$, Bachelis spaces $N_p$ $(1 < p < \infty)$ \cite{CurberaRicker2015}, $bv_0$, and $bv$ \cite{Okutoyi1990, Okutoyi1992}, as well as in weighted $\ell_p$ spaces \cite{AlbaneseBonetRicker2018, AlbaneseBonetRicker2015_Discrete}, discrete Cesàro spaces $\operatorname{ces}_p$ $(p \in \{0\} \cup (1, \infty))$ \cite{CurberaRicker2014}, and their dual spaces $d_s$ $(1 < s < \infty)$ \cite{BennettSharpley1988}.

Subsequently, in 2014, C. Lizama, P. J. Miana, R. Ponce and L. Sánchez-Lajusticia studied in \cite{LizamaMianaPonceSanchez-Lajusticia2014} a generalized Cesàro-Hardy operator. For $\alpha > 0$, they considered the Cesàro-Hardy operator of order $\alpha$ defined by
\[
    \Ca f (s) = \frac{\alpha}{s^\alpha} \int_0^s (s-u)^{\alpha-1} f(u) \, du, \quad s \in (0,\infty),
\]
on certain weighted Sobolev spaces $\mathscr{T}_p^\beta(t^\beta)$, where $p \geq 1$ and $\beta > 0$. The space $\mathscr{T}_p^\beta(t^\beta)$ consists of functions on the positive half-line $(0,\infty)$ such that $t \to t^\beta W^\beta_+f(t)$, where $W^\beta_+f$ is the Weyl fractional derivative of $f$ of order $\beta$, belongs to $L^p(0,\infty)$. Note that these spaces are contained within the Lebesgue spaces $L^p(0,\infty)$ (they coincide for $\beta=0$). The significance of this work lies not in the generalization of the operator itself, but in the approach employed. The introduction of $C_0$-semigroup theory to this problem offers a clearer understanding. Specifically, the approach involved representing the generalized Cesàro-Hardy operator as an operator subordinated via a $C_0$-group, allowing the spectral analysis of the (semi)group's generator to be applied directly to the operator itself. This work can be viewed as a natural continuation of Boyd's work \cite{Boyd1968}. Later, the discrete case was revisited within the $C_0$-semigroup framework in \cite{AbadiasMiana2018} extending Leibowitz's work \cite{Leibowitz72}. Therefore, the only case left was the study of this generalized Cesàro-Hardy operator on the unit interval using $C_0$-semigroup techniques, which is where this work starts. 

Let $\Real \alpha > 0$, consider the Cesàro-Hardy operator of order $\alpha$ given by
\[
    \Ca f (s) = \frac{\alpha}{s^\alpha} \int_0^s (s-u)^{\alpha-1} f(u) \, du, \quad s \in (0,1),
\]
and the dual Cesàro-Hardy operator of order $\alpha$ given by
\[
    \Ca^* f (s) = \alpha \int_s^1 \frac{(u-s)^{\alpha-1}}{u^\alpha} f(u) \, du, \quad s \in (0,1),
\]
which will be key to prove some spectral properties. Note that $\mathcal{C}_{1}=C_1.$ For $ \Ca,$ the change of variables $u = e^{-t}s$ yields
\[
    \Ca f (s) = \alpha \int_0^\infty (1-e^{-t})^{\alpha-1} e^{-t}f(e^{-t} s) \, dt,
\]
and for $\Ca^*$ the change of variable $u = e^{t}s$ gives
\[
    \Ca^* f (s) = \alpha \int_0^{\log(1/s)} (1-e^{-t})^{\alpha-1}f(e^t s) \, dt =  \alpha \int_0^\infty (1-e^{-t})^{\alpha-1} \chi_{(0,e^{-t})}(s) f(e^t s) \, dt,
\]
where $\chi_{I}(s)$ denotes the characteristic function of the interval $I$. The operator families, $(T(t))_{t \geq 0}$ and $(S(t))_{t \geq 0}$ given by
\begin{equation}
    \label{eq:semigroups}
    T(t)f(s) := e^{-t} f(e^{-t}s), \quad S(t)f(s) := \chi_{[0,e^{-t})}(s) f(e^t s), \quad s \in [0,1], \, t \geq 0,
\end{equation}
are $C_0$-semigroups on $L^p[0,1]$ for $1 \leq p < \infty$, see \Cref{th:semigroups}. Both semigroups are referred to in the literature as Koopman or composition semigroups. Next, with this representation in mind, we obtain the fine spectrum of both infinitesimal generators and transfer their spectral properties, first to the semigroup, using standard $C_0$-semigroup theory, see \Cref{th:fine_spectrum_T,th:fine_spectrum_S}. And later on, to the Cesàro-Hardy operator using the Hille-Phillips functional calculus and the functional calculus for sectorial operators developed by M. Haase in \cite{Haase2006}, see \Cref{th:spectrum_C*,th:spectrum_C}. This semigroup approach allows us to get, in a unified way, the spectrum of wide class of integral operators, see \Cref{sec:spectra_cesaro}.

As a consequence of the spectral properties, we prove in \Cref{th:T_is_universal} that for $t>0$ and $\mu \in B(0,e^{-t/2})$, the open ball centered at $0$ with radius $e^{-t/2}$, the operator $\mu - T(t)$ is universal on $L^2[0,1]$ in the sense of Rota, see \cite{Rota1959,Rota1960}. This notion of universality is intimately connected to the well-known Invariant Subspace Problem, a relationship that we clarify in the following lines. In an infinite-dimensional separable complex Hilbert space $H$, a bounded linear operator $U$ is universal if for every operator $T$ in $\BB(H)$ (the algebra of bounded linear operators acting on $H$) there exist a scalar $\lambda$, a closed invariant subspace $M$ of $U$, and an isomorphism $\Theta: H \to M$ such that
\[
    \Theta T = \lambda U \mid_M \Theta.
\]
Essentially, $T$ is similar to a scaled restriction of $U$ to one of its invariant subspaces. For a universal operator $U$, the following statements are equivalent:
\begin{itemize}
    \item[(i)] every bounded linear operator on an infinite-dimensional separable Hilbert space has a non-trivial closed invariant subspace;
    \item[(ii)] every minimal non-trivial invariant subspace of $U$ is one-dimensional.
\end{itemize}
S.~R. Caradus \cite{Caradus1969} showed that if an operator $T$ on a separable Hilbert space has an infinite-dimensional kernel and is surjective, then $T$ is universal. We apply this result to prove the universality of the translations of $T(t)$ in \Cref{th:T_is_universal}.

Recently, J. Agler and J. E. McCarthy have characterized the invariant subspaces of $\mathcal{C}_1$ on $L^2[0,1]$. They prove that the lattice of non-zero invariant subspaces of $\mathcal{C}_1$ on $L^2[0,1]$ (denoted by $\lat \mathcal{C}_1$) is the set of monomial spaces (see more details about this concept in \Cref{sec:universality_lattice}). We prove in \Cref{th:characterization_lattice} that the lattice of closed invariant subspaces of the semigroup $(T(t))_{t>0}$ coincides with that of $\mathcal{C}_1$ on $L^2[0,1]$, that is, $\cap_{t>0}\lat T(t)=\lat \mathcal{C}_1$.

Several recent works have also addressed invariant subspaces of Cesàro-type operators. In the complex-variable setting, \cite{Gallardo-GutierrezPartington2024,Gallardo-GutierrezPartingtonRoss2026} study invariant subspaces of the classical Cesàro operator on Hardy spaces of the unit disk, making use of semigroup techniques (namely the common invariant subspaces of a semigroup are the same as the invariant subspaces of the related Cesaro/Hardy operator). On the real-variable side, \cite{Gallardo-GutierrezPartingtonRoss2025} studies the classical Hardy operators on $L^2(0,1)$ and $L^2(0,\infty)$, with particular emphasis on cyclic and $\ast$-cyclic vectors and invariant subspaces. Very recently, \cite{BelliGulRossSiskakis2026} studied the real-variable Cesàro operator
\[
    \mathfrak{C}_p f(x):=\frac{1}{x}\int_0^x \frac{f(u)}{1-u}\,du
\]
on $L^p(0,1)$, again using related weighted composition semigroup techniques. Although the techniques are closely related to ours, their operator is different from the Cesàro-Hardy operator considered here.

While we initially presented the work of \cite{LizamaMianaPonceSanchez-Lajusticia2014} as the motivation for addressing the problem discussed in this paper, it is important to acknowledge that the earliest known reference establishing the connection between $C_0$-semigroups and the Cesàro-Hardy operator is the article by C.~C. Cowen \cite{Cowen1984}. Afterward, the technique was further employed by A. Arvanitidis and A. Siskakis \cite{ArvanitidisSiskakis2013}, who proved that the half-plane versions of Cesàro-Hardy operators on Hardy spaces are bounded. Similarly, W. Arendt and B. de Pagter \cite{ArendtDePagter2002} examined the Cesàro-Hardy operator defined on an interpolation space $E$ between $L^1$ and $L^\infty$ on the positive real half-line. Additionally, as mentioned before, the generalized discrete Cesàro-Hardy operator was studied in \cite{AbadiasMiana2018}, and more recently, in \cite{AbadiasOliva-Maza2024}, the authors used the $C_0$-semigroup approach to determine spectral sets of generalized Hausdorff operators on various spaces of analytic functions, including generalized Cesàro operators.

The paper is structured as follows. In \Cref{sec:semigroups}, we examine the properties of the one-parameter families $T(t)$ and $S(t)$. \Cref{sec:spectra_generators} discusses the spectral properties of their infinitesimal generators. This sets the stage for a detailed analysis of the spectrum of the semigroups and the generalized Cesàro-Hardy operator in \Cref{sec:spectra_semigroups} and \Cref{sec:spectra_cesaro}. In \Cref{sec:discrete}, we establish intertwining relations between the continuous semigroups $(T(t))_{t\geq0}$ and $(S(t))_{t\geq0}$ on $L^p[0,1]$ and the discrete semigroups on $\ell^p$ introduced in \cite{AbadiasMiana2018}. By means of the corresponding subordination formulas, these identities yield intertwining relations between the generalized Cesàro--Hardy operators $\Ca,\Ca^*$ and their discrete counterparts. The intertwining maps are not invertible, so these relations are not similarity statements and are not used to transfer spectral or invariant-subspace information. The spectral equalities recorded in that section follow independently from the spectral descriptions obtained in the continuous and discrete settings. Finally, in \Cref{sec:universality_lattice}, we study the universality of translations of $T(t)$ and its connection with the Invariant Subspace Problem. We also investigate the invariant subspaces of the semigroup $(T(t))_{t>0}$ and show, in the classical case, that its common invariant subspaces coincide with those of the Cesàro-Hardy operator $\mathcal C$. In addition, we obtain results concerning the cyclic subspaces of the universal operators $T(t)$.

To conclude this introduction, we recall some classical definitions. Let $X$ be a Banach space, $\BB(X)$ will denote the Banach algebra of bounded linear operators on $X$ with the operator norm $\norm{\cdot}$. Let $1 \leq p \leq \infty$, the Lebesgue spaces $L^p[0,1]$ are Banach spaces with the $L^p$-norm denoted by $\norm{\cdot}_p$. For $1 \leq p < \infty$ the dual space of $L^p$ is $L^{p'}$, where $p'$ denotes the conjugate exponent of $p$. For $1 < p < \infty$, $p'$ satisfies $\frac{1}{p} + \frac{1}{p'} = 1$. If $p = 1$, then $p' = \infty$ and $\frac{1}{p'}$ is understood as $0$. Throughout the paper we will maintain this notation.

Lastly, we recall some classical definitions in spectral theory. For a closed linear operator $\mathcal{A}:D(\mathcal{A})\subseteq X \to X$ on a Banach space $X$, where $D(\mathcal{A})$ denotes the domain of the operator, we denote the spectrum by $\sigma(\mathcal{A})$, the point spectrum by $\sigma_{point}(\mathcal{A})$, the approximate point spectrum by 
\[
    \sigma_{ap}(\mathcal{A}) := \{ \lambda \in \C \, : \, (\lambda - \mathcal{A}) \text{ is not injective or } \ran(\lambda - \mathcal{A}) \text{ is not closed in } X \},
\]
the residual spectrum by
\[
    \sigma_r(\mathcal{A}) := \{ \lambda \in \C \, : \, \ran(\lambda - \mathcal{A}) \text{ is not dense in } X \},
\]
and the essential spectrum by 
\[
    \sigma_{ess}(\mathcal{A}) := \{ \lambda \in \C \, : \, \dim(\ker(\lambda - \mathcal{A})) = \infty \text{ or } \codim(\ran(\lambda - \mathcal{A})) = \infty \}
\]
where $\codim(Y) := \dim(X/Y)$ for any subspace $Y \subseteq X$. The spectral radius $r(\mathcal{A})$ of a bounded operator $\mathcal{A}$ is given by $r(\mathcal{A}) :=\sup\{\abs{\lambda} \, : \, \lambda \in \sigma(\mathcal{A})\}$.

\section{Weighted Koopman semigroups on \texorpdfstring{$L^p[0,1]$}{Lᵖ[0,1]}} \label{sec:semigroups}

In this section we study, for $1 \leq p \leq \infty,$ the $C_0$-semigroup structure of the two operator families $(T(t))_{t \geq 0}$ and $(S(t))_{t \geq 0}$ on the spaces $L^p[0,1]$ defined on \eqref{eq:semigroups}. Along with this study, several results about the integral operators, $\Lambda_0^{\lambda}, \Lambda_1^{\lambda}$ and $L^{\lambda}$ defined on \eqref{eq:integral_Lambda_operators} and \eqref{eq:integral_L_operator} are stated. The ideas of this first result are similar to the ones present in \cite{AbadiasMiana2018} so we omit the proof. 

\begin{proposition}
    \label{pr:boundedness_semigroup}
    Let $1 \leq p \leq \infty$ and $t \geq 0$. Then:
    \begin{enumerate}[(i)]
        \item The operator $T(t)$ is bounded on $L^p[0,1]$ and $\norm{T(t)} = e^{-t/p'}$.
        \item The operator $S(t)$ is bounded on $L^p[0,1]$ and $\norm{S(t)} = e^{-t/p}$.
    \end{enumerate}
\end{proposition}

\begin{theorem}
    \label{th:semigroups}
    Let $1 \leq p < \infty$. The one-parameter operator families $(T(t))_{t \geq 0}$ and $(S(t))_{t \geq 0}$ are $C_0$-semigroups on $L^p[0,1]$ with growth bounds $-1/p'$ and $-1/p$, respectively. Their infinitesimal generators, $A$ and $B$, are given by
    \[
        Af(s)=-sf'(s)-f(s), \quad \text{ and } \quad Bf(s) = sf'(s), \quad s \in [0,1],
    \]
    defined on the domains,
    \[
        D(A) := \{f \in L^p[0,1], f \in AC_{loc}(0,1) \text{ and } sf'(s) \in L^p[0,1]\},
    \]
    and
    \[
        D(B) := \{f \in L^p[0,1], f \in AC_{loc}(0,1)\text{, } \lim_{s \to 1^-} f(s) = 0 \text{ and } sf'(s) \in L^p[0,1]\},  
    \]
    where $AC_{loc}(0,1)$ is the space of locally absolutely continuous functions on $(0,1)$.
\end{theorem}
\begin{proof}

    Firstly, note that both semigroups satisfy the semigroup property and are bounded operators for each $t \geq 0$ as shown in \Cref{pr:boundedness_semigroup}. For strong continuity, this condition is fulfilled due to the dominated convergence theorem. 
    
    For $(T(t))_{t\geq0}$, differentiating its explicit formula at $t=0$ yields
    \[
        Af(s)=-sf'(s)-f(s),
    \]
    with 
    \[
        D(A)=\{f\in L^p[0,1]: f\in AC_{\mathrm{loc}}(0,1),\ sf'(s)\in L^p[0,1]\}.
    \]
    For $(S(t))_{t\geq0}$, the result follows from \cite[Remark 2.3]{ArendtDePagter2002} by restricting the corresponding semigroup on $L^p(0,\infty)$ to functions supported on $[0,1]$ and identifying this space isometrically with $L^p[0,1]$. The local absolute continuity of the extension by zero across $s=1$ yields the condition $\lim_{s\to1^-}f(s)=0$. Hence, \[ Bf(s)=sf'(s), \] with \[ D(B)=\{f\in L^p[0,1]: f\in AC_{\mathrm{loc}}(0,1),\ \lim_{s\to1^-}f(s)=0,\ sf'(s)\in L^p[0,1]\}. \] Finally, the growth bounds $-1/p'$ and $-1/p$ follow from \Cref{pr:boundedness_semigroup}.
\end{proof}

\begin{remark}
    \label{remark:duality}
    Observe that $S(t)$ on $L^{p'}[0,1]$ is the dual operator of $T(t)$ on $L^p[0,1]$ for $1 \leq p <\infty$.
\end{remark}    

\begin{remark}
    Notice that the semigroups $(T(t))_{t\geq 0}$ and $(S(t))_{t\geq 0}$ are not strongly continuous on $L^\infty[0,1]$ as they are not uniformly continuous, see \cite[Theorem 4]{Lotz1985}.
\end{remark}    

In the following, we introduce three integral operators which play a key role in the study of the spectral properties. Let $\lambda\in\C$ and $f$ a measurable function on $[0,1],$ we write 
\begin{equation}
    \label{eq:integral_Lambda_operators}
    (\Lambda_0^{\lambda} f)(s)=\frac{1}{s^{\lambda+1}}\int_0^s u^{\lambda} f(u)\,du,\quad (\Lambda_1^{\lambda} f)(s)=s^{\lambda}\int_s^1 \frac{f(u)}{u^{\lambda+1}}\,du,\quad s\in (0,1),
\end{equation}
and 
\begin{equation}
    \label{eq:integral_L_operator}
    L^{\lambda}f=\int_0^1 \frac{f(u)}{u^{\lambda+1}}\,du.
\end{equation}

\begin{proposition}
    \label{prop:integral_resolvent}
    Let $1\leq p<\infty$ and $f\in L^p[0,1],$ then
    \begin{enumerate}[(i)]
        \item $\displaystyle\Lambda_0^{\lambda}f=\int_{0}^{\infty}e^{-\lambda t}T(t)f\,dt$ in $L^p[0,1]$ for $\Real \lambda>-1/p'.$
        \item $\displaystyle\Lambda_1^{\lambda}f=\int_{0}^{\infty}e^{-\lambda t}S(t)f\,dt$ in $L^p[0,1]$ for $\Real \lambda>-1/p$.
    \end{enumerate}
\end{proposition}
\begin{proof}
    Let $f\in L^p[0,1]$, by Hölder's inequality one gets 
    \[
        \int_0^s |u^{\lambda} f(u)|\,du\leq \|f\|_p\biggl(\int_0^s u^{p'\Real \lambda} \, du\biggr)^{1/p'},\quad s\in (0,1),
    \]
    which converges for $\Real\lambda>-1/p',$ and then $\Lambda_0^{\lambda}f$ is a measurable function. Also, 
    \[
        \int_s^1 \abs{\frac{f(u)}{u^{\lambda+1}}}\,du\leq  \|f\|_p\biggl(\int_s^1 \frac{1}{u^{p'(\Real \lambda+1)}}\,du \biggr)^{1/p'}, \quad s\in (0,1),
    \]
    which always converges, so $\Lambda_1^{\lambda}f$ is a measurable function. Notice that for the case $p=1$ the last integral on each inequality has to be understood as the $L^\infty$ norm.

    By the growth bounds of the semigroups $(T(t))_{t\geq 0}$ and $(S(t))_{t\geq 0}$ given in Theorem \ref{th:semigroups}, the right-side hand integrals in the statements of this proposition are Bochner convergent in $L^p[0,1].$ 

    Finally, for $s\in(0,1)$, the changes of variables $u=e^{-t}s$ and $u=e^{t}s$ give
    \[
        (\Lambda_0^{\lambda}f)(s)=\int_{0}^{\infty}e^{-\lambda t}(T(t)f)(s)\,dt, \quad \text{ and } \quad  (\Lambda_1^{\lambda}f)(s)=\int_{0}^{\infty}e^{-\lambda t}(S(t)f)(s)\,dt. \qedhere
    \]
\end{proof}

Now, we turn our attention to the functional $L^\lambda.$ Depending on the values of $p$ and $\lambda$ the domain will change and also the behavior of our functional. This is shown in the following results.

\begin{lemma}
    \label{lemma:L_lambda_bounded}
    Let $1 \leq p < \infty$ and $\lambda \in \C$ with $\Real \lambda < -1/p.$ 
    Then $L^{\lambda}: L^p[0,1] \to \C$ is a continuous linear functional on $L^p[0,1].$ 
    Additionally, for $p=1$ and $\Real \lambda = -1$, $L^{\lambda}$ is also continuous.
\end{lemma}
\begin{proof}
    Let $f\in L^p[0,1]$ and $\lambda\in\C$ such that $\Real \lambda <-1/p.$ Then by Hölder's
    inequality one gets 
    \[
        \abs*{\int_0^1\frac{f(u)}{u^{\lambda+1}}\,du} \leq \norm{f}_p \left(\int_0^1\frac{1}{u^{p'(\Real \lambda +1)}} \right)^{1/p'},
    \]
    which is convergent, therefore $L^{\lambda}$ is a well-defined continuous linear functional on $L^p[0,1].$ Notice that for the case $p=1$ the last integral has to be understood as the $L^\infty$ norm and therefore for $\Real \lambda = -1$ we also have
    that $L^{\lambda}$ is a well-defined continuous linear functional on $L^1[0,1].$ 
\end{proof}

The last result leaves unexplored the case when $\Real \lambda = -1/p$. In particular,
notice that $L^\lambda$ might not even be well-defined when $p \neq 1$. Nonetheless,
we can define the functional on a dense subspace of $L^p[0,1]$, $C_c^\infty(0,1)$ the space of infinite differentiable functions of compact support on $(0,1)$. In that case, see that the functional is unbounded.

\begin{lemma}
    \label{lemma:L_lambda_unbounded}
    Let $1 < p < \infty$ and $\lambda \in \C$ with $\Real \lambda = -1/p.$ Then $L^{\lambda}: C_c^\infty(0,1) \to \C$ is an unbounded linear functional with respect to the norm $\|\cdot\|_p.$
\end{lemma}
\begin{proof}
    First, notice that $L^\lambda$ is well-defined for all the functions in $C_c^\infty(0,1)$ as they have compact support. Using smooth cut-off functions, we construct a sequence of functions in $C_c^\infty(0,1)$ that will show that the functional is unbounded. Let $\eta,\zeta\in C^\infty(\R)$ satisfy
    $0\leq\eta,\zeta\leq1$, such that
    \[
    \eta(t) =
    \begin{cases}
         0, & t \leq 1, \\ 1, & t \geq 2,
    \end{cases}
    \qquad
    \zeta(u) =
    \begin{cases}
         1, & u \leq 1/4, \\ 0, & u \geq 1/2.
    \end{cases}
    \]
    For $n\in\N$ with $n > 8$, we write $\varphi_n(u) = \eta(nu)\zeta(u)$. So $\varphi_n \in C_c^\infty(0,1)$, $0 \leq \varphi_n \leq 1$, $\varphi_n(u)=1$ for $u \in [2/n,1/4]$, and $\varphi_n(u)=0$ if $u \in [0,1/n] \cup [1/2,1]$. Now define the sequence
    \[
        f_n(u) = \varphi_n(u) \frac{u^\lambda}{\log(1/u)^\frac{1}{p}}, \quad u\in(0,1), \; n>8.
    \]
    On one hand, we have that
    \[
        L^\lambda f_n =     \int_0^1 \frac{\varphi_n(u)}{u \log(1/u)^\frac{1}{p}} \, du
                      \geq  \int_\frac{2}{n}^\frac{1}{4} \frac{du}{u \log(1/u)^\frac{1}{p}}
                      = \int_{\log 4}^{\log \frac{n}{2}} s^{-\frac{1}{p}}\, ds 
                      = p' \left( \left(\log \frac{n}{2}\right)^\frac{1}{p'}\!\!- (\log 4)^\frac{1}{p'}\right).
    \]
    On the other hand,
    \[
        \norm{f_n}_p^p  \leq \int_\frac{1}{n}^\frac{1}{2} \abs*{\frac{u^\lambda}{\log(1/u)^\frac{1}{p}}}^p \, du 
                        = \int_\frac{1}{n}^\frac{1}{2} \frac{du}{u \log(1/u)}
                        = \int_{\log 2}^{\log n} \frac{1}{s} \, ds
                        = \log (\log n )  -  \log (\log 2 ).
    \]
    Therefore,
    \[
        \frac{\abs{L^\lambda f_n}}{\norm{f_n}_p}
        \geq
        \frac{p' \left( \left(\log \frac{n}{2}\right)^\frac{1}{p'}- (\log 4)^\frac{1}{p'}\right)}
        {\left(\log(\log n)-\log(\log 2)\right)^\frac{1}{p}}
        \longrightarrow +\infty.
    \]
    So, the functional $L^\lambda$ is unbounded.
\end{proof}

\begin{remark}
    \label{remark:ker_functional}
    Note that for $1\leq p<\infty$ and $\Real \lambda \leq -1/p,$ the kernel of the functional $L^{\lambda}$ on $C_c^\infty(0,1)$ is not the whole domain as it's not the zero functional.
\end{remark}

\begin{corollary}
    \label{kerdensity}
    For $1 < p <\infty$ and $\Real \lambda = -1/p,$ the kernel of the functional $L^{\lambda}$ on $C_c^\infty(0,1)$ is dense on $L^p[0,1].$\end{corollary}
\begin{proof}
    Recall that a non-zero linear functional on a topological vector space $X$, which is the case for $L_\lambda$ on $C_c^\infty(0,1)$ with the norm $\|\cdot\|_p$, is continuous if and only if the kernel is not dense in $X$, see \cite[Theorem 1.18]{Rudin1991}. So, the result follows by \Cref{lemma:L_lambda_unbounded} and the density of $C_c^\infty(0,1)$ in $L^p[0,1]$ with the norm $\|\cdot\|_p.$
\end{proof}

\section{Spectra of differential operators} 
\label{sec:spectra_generators}

In this section we present the fine spectrum of the infinitesimal generators $A$ and $B$ on $L^p[0,1].$ 
\begin{proposition}
    \label{prop:point_spectrum}
    Let $1\leq p<\infty.$ The point spectrum of $B$ is empty, that is, $\sigma_{point}(B)=\emptyset.$ The point spectrum of $A$ is the set $\sigma_{point}(A)=\{\lambda\in\C\,:\, \Real\lambda<-1/p'\}.$ Furthermore, the eigenspace of each $\lambda\in\sigma_{point}(A) $ is one-dimensional and generated by $g_{\lambda}(s):=\frac{1}{s^{\lambda+1}}.$
\end{proposition}
\begin{proof}
    Let $\lambda \in \C$ and $f \in D(B)$ such that $Bf = \lambda f$. Then, $f$ is a solution of the differential equation
    \[
        sf'(s)-\lambda f(s) = 0.
    \]
    The nonzero solutions to this equation have the form $f(s) = c s^{\lambda}$ with $c \in\C$ and $c\neq 0,$ but these functions are not in $D(B)$ because $\lim_{s \to 1^-} f(s) \neq 0.$ Therefore, $\sigma_{point}(B)=\emptyset.$

    If $\lambda \in \C$ and $f \in D(A)$ such that $Af = \lambda f$, then $f$ satisfies
    \[
        (\lambda + 1)f(s) + sf'(s) = 0.
    \]
    The solutions are of the form $f(s) = c g_{\lambda}(s)$, where $c$ is a complex constant. These solutions belong to $L^p[0,1]$ if and only if $c = 0$ or $\Real \lambda < -1/p'$. Therefore, $\sigma_{point}(A) = \{\lambda \in \C : \Real \lambda < -1/p'\}$, and the associated eigenspace is the one-dimensional subspace generated by $g_{\lambda}$.
\end{proof}

\begin{proposition}
    \label{prop:spectrum_A} 
    Let $1\leq p<\infty.$ The  spectrum of $A$ is the set $\sigma(A)=\{\lambda\in\C\,:\, \Real\lambda\leq -1/p'\}.$ Moreover, for $\lambda\in\C$ with $\Real\lambda>-1/p',$ $$(\lambda-A)^{-1}f=\Lambda_0^{\lambda}f,\quad f\in L^p[0,1].$$
\end{proposition}
\begin{proof}
    Let $\Real\lambda >-1/p'$, by the integral representation of the resolvent operator of $A$ in terms of the semigroup $(T(t))_{t\geq 0}$ (see e.g. \cite[Th. II.1.10]{EngelNagel2000}) and Proposition \ref{prop:integral_resolvent}, one has
	\[
        (\lambda-A)^{-1}f = \Lambda_0^{\lambda}f, \quad f\in L^p[0,1].
    \]

    Therefore, $\sigma(A) \subseteq \{\lambda \in \C : \Real \lambda \leq -1/p'\}$. Since $\sigma_{point}(A) \subseteq \sigma(A)$ and the spectrum is a closed subset of $\C,$ the result follows by Proposition \ref{prop:point_spectrum}.
\end{proof}

\begin{proposition}
    \label{prop:spectrum_B} 
    Let $1\leq p<\infty.$ The  spectrum of $B$ satisfies $\sigma(B)\subseteq \{\lambda\in\C\,:\, \Real\lambda\leq -1/p\}.$ Moreover, for $\lambda\in\C$ with $\Real\lambda>-1/p,$ $$(\lambda-B)^{-1}f=\Lambda_1^{\lambda}f,\quad f\in L^p[0,1].$$
\end{proposition}
\begin{proof}
    The result is a straightforward consequence of  \cite[Th. II.1.10]{EngelNagel2000} and Proposition \ref{prop:integral_resolvent}.
\end{proof}

\begin{remark} 
    As we have seen in \Cref{remark:duality} the semigroup $S(t)$ is the dual semigroup
    of $T(t)$ for each $t \geq 0$. Notice that as the space $L^p[0,1]$ is reflexive for 
    $1 < p < \infty$ using $C_0$-semigroup theory and duality we could compute the exact spectrum directly of the operator $B$, see \cite[Paragraph II.2.5]{EngelNagel2000}.
    But, in order to do a detailed study we leave this as a remark.
\end{remark}

We proceed to study the range space of the operator $\lambda - A: D(A) \to L^p[0,1]$ for a fixed $\lambda \in \C$.  Given $f\in L^p[0,1],$ it is an easy computation that the solutions of the differential equation  $(\lambda-A)g= f$ are given by
\begin{equation}
    \label{ODEsol}
		g(s)=g_{f,K}(s) = \frac{1}{s^{\lambda+1}}\biggl(K-\int_s^1 u^{\lambda} f(u)\,du\biggr), \quad  s \in (0,1),\,  K\in \C.
\end{equation}

\begin{lemma}
    \label{rangeA} 
    Let $1\leq p<\infty$ and $\lambda\in\C$ with $\Real \lambda <-1/p'.$ Then $\ran (\lambda-A)=L^p[0,1].$
\end{lemma}
\begin{proof}
    On one hand notice that $\frac{1}{s^{\lambda+1}} \in L^p[0,1]$ if and only if $\Real \lambda <-1/p'.$ Therefore, we have to check what happens with the integral part of \eqref{ODEsol}. The change of variable $u=e^t s$ gives
    \[
        \frac{1}{s^{\lambda+1}} \int_s^1 u^\lambda f(u) \, du = \int_0^\infty e^{t(\lambda+1)} (S(t)f)(s) \, dt,\quad s\in (0,1).
    \] 
    We have already seen in the proof of Theorem \ref{th:semigroups} that $\| S(t) f\|_{p} = e^{-\frac{t}{p}} \|f\|_p$. Therefore,
    \[
        \int_0^\infty |e^{t(\lambda+1)}|  \|S(t)f \|_p \, dt = \| f \|_p \int_0^\infty e^{t(\Real \lambda +1 -\frac{1}{p})} \, dt,
    \]
    where the integral converges for $\Real \lambda < -\frac{1}{p'}$. Thus, $\ran(\lambda - A) = L^p[0,1]$ for all $\lambda \in \C$ such that $\Real \lambda < - \frac{1}{p'}$ as claimed.
\end{proof}

\begin{remark} Note that by Proposition \ref{prop:spectrum_A} and Lemma \ref{rangeA} the range of $\lambda-A$ is the whole space $L^p[0,1]$ ($1\leq p<\infty$) for $\Real \lambda\neq -1/p'.$ Next result states what happens when  $\Real \lambda =-1/p'.$
\end{remark}

\begin{lemma}
    \label{lemma:rangeA_density}
        Let $1\leq p<\infty$ and $\lambda\in\C$ with $\Real \lambda =-1/p'.$ Then $\ran (\lambda-A)$ is dense on $L^p[0,1]$.
\end{lemma}
\begin{proof}
    We will use the fact that the functions with compact support, $C_c^\infty(0,1)$ are dense in $L^p[0,1]$. Let  $\lambda\in\C$ with $\Real \lambda =-1/p',$ $f \in C_c^\infty(0,1),$ and define
    \[
        g(s) = \frac{1}{s^{\lambda+1}}\biggl(\int_0^1 u^{\lambda} f(u)-\int_s^1 u^{\lambda} f(u)\,du\biggr),\quad s\in (0,1).
    \]
    Notice that $g$ is bounded, and therefore $g \in L^p[0,1].$ Moreover, by \eqref{ODEsol} $g$ is a solution of the differential equation $(\lambda-A)g= f$ with  $g \in D(A).$ This implies that $ C_c^\infty(0,1) \subseteq Ran(\lambda - A)$. As $C_c^\infty(0,1)$ is dense in $L^p[0,1]$ we have that $Ran(\lambda - A)$ is dense as well in $L^p[0,1]$.
\end{proof}

Now we focus on the study of the range space of the operator $\lambda - B: D(B) \to L^p[0,1]$ for a fixed $\lambda \in \C.$ Given $f \in L^p[0,1],$ the solutions of the differential equation $(\lambda-B)h= f$ are given by

\begin{equation}
    \label{eq:ODEsol2}
    h(s)=h_{f,K}(s) = s^{\lambda}\biggl (K+\int_s^1 \frac{f(u)}{u^{\lambda+1} }\,du\biggr), \quad  s \in (0,1),\,  K\in \C.\\
\end{equation}

\begin{lemma}
    \label{lemma:rangeB} 
   Let $1 \leq p < \infty$ and $\lambda \in \C$ with $\Real \lambda < -1/p$. Then, $\ran (\lambda - B) = \ker L^\lambda$. Additionally, for $p = 1$ and $\Real \lambda = -1$, we have $\ran (\lambda - B) \subseteq \ker L^\lambda$.

\end{lemma}
\begin{proof}
   
We first start proving that $\ran(\lambda - B) \subseteq \ker L^\lambda$ for $\Real \lambda<-1/p$ with $1<p<\infty,$ and also for $\Real \lambda\leq-1$ if $p=1.$ The idea will be to take a function in $\ran(\lambda - B)$ that is not in $\ker L^\lambda$, and we will reach a contradiction.

First note that the mappings,
    \[
        \begin{aligned}
           I_1: [0,1] & \to \C\\
            s &\mapsto \int_s^1 \frac{f(u)}{u^{\lambda+1} }\,du, \quad f \in L^p[0,1],
        \end{aligned}
    \]
   and  \[\begin{aligned}
           I_2: [0,1] & \to \C\\
            s &\mapsto \int_0^s \frac{f(u)}{u^{\lambda+1} }\,du, \quad f \in L^p[0,1],
        \end{aligned} \]
        are continuous  by under some assumptions on the value $\lambda,$ because \Cref{lemma:L_lambda_bounded} allows to apply  the Dominated Convergence Theorem.

  Let now $f \in \ran(\lambda - B)$ with $f \notin \ker L^\lambda$ and let $h$ given by \eqref{eq:ODEsol2}. We have two different scenarios, $K = - L^\lambda f$ or $K \neq - L^\lambda f$. In the first case, the function $h$ simplifies to,
    \[
        h(s) = -s^\lambda \int_0^s \frac{f(u)}{u^{\lambda+1}}\,du,
    \]
    and by the continuity of $I_2$ one gets $\lim_{s \to 1^{-}} h(s) = - L^\lambda f \neq 0$ since $f \notin \ker L^\lambda,$ hence $h \notin D(B)$, which is a contradiction. On the second scenario, by the continuity of $I_1$ we have $\lim_{s \to 0^+} h(s) = \lim_{s\to 0^+} c s^\lambda$ for some complex constant $c \neq 0$. Since $s^\lambda\notin L^p[0,1]$ for $\Real \lambda<-1/p$ with $1<p<\infty,$ and also for $\Real \lambda\leq-1$ if $p=1,$ then $h\notin L^p[0,1].$ So, we have reached a contradiction in either case, and we have proved that $\ran(\lambda - B) \subseteq \ker L^\lambda.$

    In the following lines we prove that $\ker L^\lambda\subseteq \ran(\lambda - B)$ for $\Real \lambda <-1/p$ with $1\leq p<\infty$. Let $f\in\ker L^\lambda.$ We have that $f$ lies in $\ran (\lambda - B)$ if 
    and only if there exists $K\in\C$ such that $h$ given by \eqref{eq:ODEsol2} is on $L^p[0,1]$. Taking $K=L^\lambda f=0$ as before, we have
    \[
        h(s) = -s^\lambda \int_0^s \frac{f(u)}{u^{\lambda +1}} \, du.
    \]
    By a change of variable,
    \[
        h(s)   =- s^\lambda \int_0^s \frac{f(u)}{u^{\lambda +1}} \, du
                    = -\int_0^\infty e^{\lambda t} f(e^{-t}s) \, dt 
                    = -\int_0^\infty  e^{(\lambda + 1) t} (T(t)f)(s) \, dt.
    \]
    Thus, applying Minkowski's inequality and the fact that $\norm{T(t)} = e^{-t/p'}$ we get
    \[
        \|h\|_p    \leq \int_0^\infty e^{(\Real \lambda +\frac{1}{p}) t} \, dt,
    \]
    which is convergent for $\Real \lambda < -\frac{1}{p}$, so $h \in L^p[0,1]$. Then $f \in \ran(\lambda - B),$ and therefore $\ran(\lambda - B) = \ker L^\lambda$. 
\end{proof}

\begin{lemma}
    \label{lemma:rangeB_density} 
    Let $1\leq p<\infty$ and $\lambda\in\C$ with $\Real \lambda \leq -1/p,$ then the $\ran (\lambda-B)$ has the following density properties
    \begin{enumerate}[(i)]
        \item if $p=1$, $\ran (\lambda-B)$ is not dense on $L^1[0,1]$,
        \item and if $p>1$, $\ran (\lambda-B)$ is dense on $L^p[0,1]$ when $\Real \lambda = -1/p$ and is not dense when $\Real \lambda < -1/p$. 
    \end{enumerate}
\end{lemma}
\begin{proof}
    For the case $p=1,$ notice that by \Cref{lemma:rangeB} $\ran (\lambda-B) \subseteq \ker L^\lambda$, where $\ker L^\lambda$ is closed since $L^\lambda$ is a bounded operator on $L^p[0,1]$ (see \Cref*{lemma:L_lambda_bounded}). Moreover, as $\ker L^\lambda \neq L^1[0,1]$, see \Cref{remark:ker_functional}, we have that
    \[
        \overline{\ran (\lambda-B)} \subseteq \ker L^\lambda \neq L^1[0,1],
    \]
    so $\ran (\lambda-B)$ is not dense. 

    On the other hand, when $p>1$ the same idea works when $\Real \lambda < -1/p$ to prove $\ran (\lambda-B)$ is not dense on $L^p[0,1]$. When $\Real \lambda = -1/p,$ the proof of \Cref{lemma:L_lambda_unbounded} shows that the linear functional $L^\lambda$ is defined on $C_c^\infty(0,1).$ Take  $f \in C_c^\infty(0,1)$  such that $L^{\lambda} f=0.$  We construct the function $h(s) = h_{f,0}(s)$ as in \eqref{eq:ODEsol2}, that is,
    \[
        h(s) = s^\lambda \int_s^1 \frac{f(u)}{u^{\lambda+1}} \, du,\quad s\in (0,1).
    \]
    Notice that $h \in L^p[0,1]$ as it is a continuous function with compact support on $(0,1).$ Then the kernel of the functional $L^{\lambda}$ on $C_c^\infty(0,1)$ is contained on $\ran(\lambda-B),$ and also is dense in $L^p[0,1]$ (see \Cref{kerdensity}). Therefore, the result follows.
\end{proof}

\begin{theorem}
    \label{th:fine_spectrum_A}
    Let $1\leq p<\infty.$ The differential operator $A$ satisfies that 
    \begin{enumerate}[(i)]
        \setlength\itemsep{0.5em}

        \item $\sigma(A)=\{\lambda\in\C\,:\, \Real\lambda\leq -1/p'\}$.

        \item $\sigma_{point}(A)=\{\lambda\in\C\,:\, \Real\lambda< -1/p'\}$. 

        \item $\sigma_{ap}(A)=\{\lambda\in\C\,:\, \Real\lambda\leq  -1/p'\}$.

        \item $\sigma_r(A)=\emptyset$.
        
        \item $\sigma_{ess}(A)=\{\lambda\in\C\,:\, \Real\lambda= -1/p'\}$.
    \end{enumerate}
\end{theorem}
\begin{proof} 
    Items (i) and (ii) are Propositions \ref{prop:spectrum_A} and \ref{prop:point_spectrum}, respectively.

    Item (iii) follows from the fact that $\sigma_{point}(A) \subseteq \sigma_{ap}(A)$ and
    $\partial \sigma(A) \subseteq \sigma_{ap}(A),$ and item (iv) follows from \Cref{rangeA} and \Cref{lemma:rangeA_density}.

    Lastly, by \Cref{prop:point_spectrum} and \Cref{rangeA} we have $\sigma_{ess}(A)\subseteq\{\lambda\in\C\,:\, \Real\lambda=-1/p'\}.$ Moreover, by Proposition \ref{prop:spectrum_A}, if $\lambda\in\C$ with $ \Real\lambda=-1/p',$ then it is an accumulation point of both the resolvent set $\rho(A)$ and the spectrum $\sigma(A),$ therefore $\lambda\in\sigma_{ess}(A)$ (see \cite[Th. I.3.25]{EdmundsEvans1987}), and the result for the essential spectrum follows.
\end{proof}

\begin{theorem}
    \label{th:fine_spectrum_B}
    Let $1\leq p<\infty.$ The differential operator $B$ satisfies that 
    \begin{enumerate}[(i)]
        \setlength\itemsep{0.5em}

        \item $\sigma(B)=\{\lambda\in\C\,:\, \Real\lambda\leq -1/p\}.$

        \item $\sigma_{point}(B)=\emptyset.$

        \item $\sigma_{ap}(B)=\{\lambda\in\C\,:\, \Real\lambda =  -1/p\}.$

        \item For $p=1$, $\sigma_r(B)=\{\lambda\in\C\,:\, \Real\lambda \leq -1\},$ \\ for $p>1$, $\sigma_r(B)=\{\lambda\in\C\,:\, \Real\lambda < -1/p\}.$
        
        \item $\sigma_{ess}(B)=\{\lambda\in\C\,:\, \Real\lambda= -1/p\}$.
    \end{enumerate}
\end{theorem}
\begin{proof}
    Item (ii) and (iv) are direct consequences of \Cref{prop:point_spectrum} and \Cref{lemma:rangeB_density}, and therefore item (i) follows by \Cref{prop:spectrum_B} and the fact that the spectrum is closed.

    For item (iii) we have that $\{\lambda\in\C\,:\, \Real\lambda =  -1/p\}\subseteq \sigma_{ap}(B)$ as the boundary of the spectrum is contained in the approximate point spectrum. Now, if $\lambda \in \C$ with $\Real\lambda < -1/p$ we have that $(\lambda - B)$ is injective (see \Cref{prop:point_spectrum}) and $\ran(\lambda-B) = \ker L^\lambda$ (see \Cref{lemma:rangeB}) which is closed, therefore $\lambda \notin \sigma_{ap}(B)$ and the result follows.

    Finally, for item (v), let  $\lambda\in\C$ with $\Real\lambda<-1/p$. By the Banach isomorphism theorem \cite[Theorem 1.7.14]{Megginson1998}  and \Cref{lemma:rangeB} we have that
    \[
        \quotient{L^p[0,1]}{\ran(\lambda - B)} = \quotient{L^p[0,1]}{\ker L^\lambda} \cong \ran L^\lambda.
    \]
    Therefore, $\codim(\lambda - B) = \dim \ran L^\lambda = 1$, because $L^\lambda$ is the non-zero linear functional. Moreover, by \Cref{prop:point_spectrum} we have that $\dim \ker (\lambda - B) = 0,$ therefore $\lambda  \notin \sigma_{ess}(B)$. Finally, by \cite[Th. I.3.25]{EdmundsEvans1987} the elements which are accumulation points of the spectrum and the resolvent set at the same time belong to the essential spectrum, so item (v) follows.
\end{proof}

\section{Spectra of Koopman semigroups}
\label{sec:spectra_semigroups}

At \Cref*{sec:semigroups} we introduced the operator families $(T(t))_{t \geq 0}$ and $(S(t))_{t \geq 0}$ which are $C_0$-semigroups for $1 \leq p < \infty,$ and which have the operators $A$ and $B$ as infinitesimal generators respectively. In this section, we use the spectral study of the generators in \Cref*{sec:spectra_generators} in order to study the spectral properties of the $C_0$-semigroups $(T(t))_{t \geq 0}$ and $(S(t))_{t \geq 0}$ defined on $L^p[0,1]$ for $1 \leq p < \infty$. From now on, we denote by $B(c,r)$ the open ball with center $c$ and radius $r$. Also, from now on, the following subspaces \[
    M_{p,r} := \{ f \in L^p[0,1] \, :\, f(s) = 0 \text{ a.e. for } s \in (0,r) \}, \quad r \in (0,1),\, 1\leq p\leq \infty,
\] will be play a key role in some main results in this paper.

\begin{theorem}
    \label{th:fine_spectrum_T}
    Let $1\leq p \leq \infty$ and $t>0$. The bounded operator $T(t)$ satisfies that 
    \begin{enumerate}[(i)]
        \setlength\itemsep{0.5em}

        \item $\sigma(T(t)) = \overline{B(0,e^{-t/p'})}$.

        \item For $ 1 \leq p < \infty$, $\sigma_{point}(T(t))=B(0,e^{-t/p'})$, \\ for $p=\infty$, $\sigma_{point}(T(t))=\overline{B(0,e^{-t})}.$

        \item $\sigma_{ap}(T(t))=\overline{B(0,e^{-t/p'})}$.

        \item $\sigma_r(T(t))=\emptyset$ for $1\leq p<\infty.$
        
        \item $\sigma_{ess}(T(t))=\overline{B(0,e^{-t/p'})}$.
    \end{enumerate}
\end{theorem}
\begin{proof}
    First, we prove the results for the case when $1 \leq p < \infty$, that is, when $(T(t))_{t \geq 0}$ is a strongly continuous semigroup.

    For item (i), we have by \cite[Theorem IV.3.6]{EngelNagel2000} that
    \[
        e^{t \sigma(A)} = \overline{B(0,e^{-t/p'})} \setminus \{0\} \subseteq \sigma(T(t))
    \]
    and as the growth bound of $T(t)$ is $w_0 = -1/p'$, we have that $r(T(t)) = e^{-t/p'}.$ Then $\sigma(T(t)) \subseteq \overline{B(0,e^{-t/p'})}$ and taking closures we get the desired result.

    From \cite[Theorem IV.3.7]{EngelNagel2000}, we have that
    \[
        \sigma_{point}(T(t)) \setminus \{0\} = B(0,e^{-t/p'}) \setminus \{0\}.
    \]
    Moreover, $0 \in \sigma_{point}(T(t))$ and item (ii) follows. Indeed, take $f = \chi_{(e^{-t},1]}$, then $T(t)f(s) = 0$ with $s\in [0,1].$ 

    Analogously, we have by \cite[Theorem IV.3.7]{EngelNagel2000} that,
    \[
        \sigma_r(T(t)) \setminus \{0\} = \emptyset.
    \]
    Even more, let $f \in L^p[0,1]$ and define $g(s) = e^t \chi_{[0,e^{-t})}(s) f(e^t s)$. It's easy to check that under a change of variables $\norm{g}_p = \norm{f}_p <\infty$, and $T(t)g(s) = f(s)$, so $\ran(T(t)) = L^p[0,1]$.
    Therefore, the range is dense and $0 \notin \sigma_r(T(t))$ and item (iv) follows.

    For the approximate point spectrum we will use the fact that $\sigma(T(t)) = \sigma_{ap}(T(t)) \cup \sigma_{r}(T(t))$, then the result follows as $\sigma_{r}(T(t)) = \emptyset$. 

    Finally, for item (v), if $\mu \in \sigma(T(t)) \setminus \{ 0\}$ with
    $\mu = e ^{\lambda t}$ with $\lambda \in \sigma_{point}(A)$ we have that the function $\frac{1}{s^{\lambda+1}}$ is an eigenfunction for $\mu$. Thus, every eigenvalue $\mu$ has infinite eigenfunctions associated with it (with periodicity $2\pi i/t$) and $\ker(\mu - T(t))$ is infinite dimensional so $\mu \in \sigma_{ess}(T(t))$. On the other hand, notice that $\ker(T(t)) =M_{p,{e^{-t}}}.$ Then $\dim(\ker(T(t))) = \infty $ and $0 \in \sigma_{ess}(T(t))$. Lastly, by \cite[Th. I.3.25]{EdmundsEvans1987} the elements which are accumulation points of $\sigma(T(t))$ and $\rho(T(t))$ at the same time belong to $\sigma_{ess}(T(t))$ and result follows.

    For $p=\infty$ we use the fact that the operator $T(t)$ defined on $L^\infty[0,1]$ is the dual operator of $S(t)$ defined on $L^1[0,1]$, therefore the results for items (i), (ii) follow from \Cref{th:fine_spectrum_S} below and \cite[Corollary B.12, Proposition IV.1.12]{EngelNagel2000}. Item (iii) follows since $\sigma_{point}(T(t))\subseteq \sigma_{ap}(T(t))\subseteq \sigma(T(t)).$  For item (v) we divide our study into 3 different cases. First the values $\lambda\in\C$ which are in the boundary of the spectrum, that is, $|\lambda|=e^{-t},$ because they are accumulation points of both the spectrum and resolvent set,  and therefore they belong to $\sigma_{ess}(T(t))$ (see \cite[Th. I.3.25]{EdmundsEvans1987}). Additionally, $0 \in \sigma_{ess}(T(t))$ because $\ker(T(t)) = M_{\infty,e^{-t}}$ which is infinite dimensional. Lastly, take $\lambda \in B(0,e^{-t}) \setminus \{ 0 \},$ then $\lambda = e^{\mu t}$ with $\mu \in \C$ with $\Real \mu< -1$. Although in this case $T(t)$ doesn't define a strongly continuous semigroup, we still have that for $n \in \Z$ the functions $g_{\lambda,n}(s) := \frac{1}{s^{\mu + \frac{2\pi i n}{t} + 1}}\in L^\infty[0,1]$ are eigenfunctions for the eigenvalue $\lambda$. Therefore, $\lambda \in \sigma_{ess}(T(t))$ since $\dim \ker (\lambda - T(t)) = \infty$. 
\end{proof}

\begin{theorem}
    \label{th:fine_spectrum_S}
    Let $1\leq p \leq \infty$ and $t>0$. The bounded operator $S(t)$ satisfies that 
    \begin{enumerate}[(i)]
        \setlength\itemsep{0.5em}

        \item $\sigma(S(t)) = \overline{B(0,e^{-t/p})}$.

        \item $\sigma_{point}(S(t))=\emptyset$. 

        \item $\sigma_{ap}(S(t))= \delta B(0,e^{-t/p})$.

        \item For $p=1$, $ \sigma_r(S(t))=\overline{B(0,e^{-t})},$ \\ for $1 < p < \infty$, $\sigma_r(S(t))=B(0,e^{-t/p}),$ \\ and for $ p = \infty$, $B(0,e^{-t/p})\subseteq \sigma_r(S(t)).$
        
        \item $\sigma_{ess}(S(t))= \overline{B(0,e^{-t/p})}$.
    \end{enumerate}
\end{theorem}
\begin{proof}
    As in the last theorem, we prove the results for the case when $1 \leq p < \infty$, that is, when $(S(t))_{t \geq 0}$ is a strongly continuous semigroup.

    Item (i) follows the same idea as in \Cref{th:fine_spectrum_T}, and for item (ii) we use \cite[Theorem IV.3.7]{EngelNagel2000} for the point spectrum and a direct computation shows that $S(t)$ is injective.

    For the residual spectrum, we have that, if $p=1$
    \[
        \sigma_r(S(t)) \setminus \{0\} = \overline{B(0,e^{-t/p})} \setminus \{0\},
    \]
    and for $p>1$,
    \[
        \sigma_r(S(t)) \setminus \{0\} = B(0,e^{-t/p}) \setminus \{0\}.
    \]
    Moreover, notice that $0 \in \sigma_r(S(t))$ in both cases ($p=1$ and $p>1$) since
    $\ran(S(t))= M_{p,e^{-t}}^{\perp}=\{f \in L^p(0,1): f(x) = 0 \text{ a.e. for } x \in (e^{-t},1) \},$ which is a closed subspace of $L^p[0,1]$ and therefore not dense.

    Next, we prove item (iii) using the characterization of the approximate point spectrum via sequences (see \cite[Lemma IV.1.9]{EngelNagel2000}). Take $\lambda \in B(0,e^{-t/p})\setminus \{0\}$ and suppose $\lambda \in \sigma_{ap}(S(t))$, then there exists a sequence $(f_n)_{n \in \N} \subset L^p[0,1]$, such that $\norm{f_n}_p=1$ and $\lim_{n \to \infty} \norm{S(t)f_n-\lambda f_n}_p = 0$. Therefore,
    \[
        \norm{S(t)f_n-\lambda f_n}_p^p = \int_0^{e^{-t}} \abs{f_n(e^t u) - \lambda f_n(u)}^p \, du + \int_{e^{-t}}^1 \abs{\lambda f_n(u)}^p \, du \to 0
    \]
    as $n \to \infty$. As both integrals are positive this means that each of them go to zero. However, doing a change of variables the first integral becomes
    \[
        \begin{aligned}
            \int_0^{e^{-t}} \abs{f_n(e^t u) - \lambda f_n(u)}^p \, du & = \int_0^1 \abs{f_n(x) - \lambda f_n(e^{-t}x)}^p e^{-t} \, dx \\ & = \abs{\lambda}^p e^{(p-1)t} \int_0^1 \abs{T(t)f_n(x) - \frac{e^{-t}}{\lambda} f_n(x)}^p  \, dx \to 0.
        \end{aligned}
    \]
    This implies that the value $\frac{e^{-t}}{\lambda} \in \sigma_{ap}(T(t))$ on the space $L^p[0,1]$. But, notice that $\abs{\frac{e^{-t}}{\lambda}} > e^{-t} e^{t/p} = e^{-t/p'}$ as $\lambda \in B(0,e^{-t/p})\setminus \{0\}$, which is a contradiction as $\sigma_{ap}(T(t)) = \overline{B(0,e^{-t/p'})}$. The fact that $0 \notin \sigma_{ap}(S(t))$ follows from the fact that $\ran(S(t)) = \{f \in L^p[0,1] \, : \, f(s)=0 \text{ a.e. for } s \in (e^{-t},1)  \}$ is closed. This case finishes with the fact that the boundary of the spectrum is always included in the approximate point spectrum.

    For item (v) we use the duality of the operators $T(t)$ and $S(t)$ and the properties of the essential spectrum and duality, see \cite[Th. IX.1.1]{EdmundsEvans1987}.
      
    The case $p=\infty$. The operator $S(t)$ defined on $L^\infty[0,1]$ is the dual of $T(t)$ defined on $L^1[0,1]$, therefore the results for items (i) and (ii) follow from \Cref{th:fine_spectrum_T} and \cite[Corollary B.12, Proposition IV.1.12]{EngelNagel2000}. Also, by the last paragraph we have also the result for item (v). The proof for the approximate spectrum is similar that for the case $1 \leq p < \infty$ using the $\|\cdot\|_{\infty}$ norm. Finally, item (iv) follows by  $\sigma(S(t)) = \sigma_{ap}(S(t)) \cup \sigma_{r}(S(t)).$
\end{proof}

\section{The fine spectra of the generalized Cesàro-Hardy operators}
\label{sec:spectra_cesaro}

In this section, we apply the results obtained in the preceding sections to study the boundedness and the spectrum on $L^p[0,1]$ of the Cesàro-Hardy operator $\Ca$ and its dual $\Ca^*$. In order to do so, we will use the theory of the natural functional calculus for sectorial operators defined by Haase in \cite{Haase2006}, but first we must fix some notation.

The open sector of angle $\omega$ is defined as $S_\omega:= \{ z \in \C \, : \, z \neq 0 \text{ and } \abs{\arg z} < \omega\}$, for $0< \omega \leq \pi$ and $S_0:=(0,\infty)$. Let $0 \leq \omega < \pi$, an operator $\mathcal{A}$ on a Banach space $X$ is called sectorial of angle $\omega$ if $\sigma(\mathcal{A})\subset \overline{S_\omega}$ and $\sup \{ \norm{\lambda (\lambda - \mathcal{A})^{-1}}  \, : \, \lambda \notin \overline{S_{\omega'}}\} < \infty$ for all $\omega < \omega' < \pi$. In particular if $-\mathcal{A}$ generates a uniformly bounded semigroup $(\mathcal{T}(t))_{t \geq 0}$, i.e. $\norm{\mathcal{T}(t)} \leq C$ for some positive constant $C$, $\mathcal{A}$ is sectorial of angle $\pi /2$ by the Hille-Yosida Theorem (\cite[Corollary II.3.6]{EngelNagel2000}). Now we are prepared to prove the results for Cesàro-Hardy operators.

Let $\Real \alpha > 0$, we consider the Cesàro-Hardy operator of order $\alpha$ given by,
\[
    \Ca f (s) = \frac{\alpha}{s^\alpha} \int_0^s (s-u)^{\alpha-1} f(u) \, du, \quad s \in (0,1),
\]
and the dual Cesàro-Hardy operator of order $\alpha$ given by
\[
    \Ca^* f (s) = \alpha \int_s^1 \frac{(u-s)^{\alpha-1}}{u^\alpha} f(u) \, du, \quad s \in (0,1).
\]
These first results show the boundedness of the Cesàro-Hardy operators and their subordination identities. The exact norm of the operators is computed in \Cref{cor:cesaro_norm}.

\begin{theorem}
    \label{th:cesaro_boundedness}
    Let $\Real \alpha > 0$. Then:
    \begin{enumerate}[(i)]
        \item The operator $\Ca$ is a bounded operator on $L^p[0,1]$, for $1 < p \leq \infty$.
        \item The operator $\Ca^*$ is a bounded operator on $L^p[0,1]$, for $1 \leq p < \infty$. 
    \end{enumerate}
    In addition, for $f \in L^p[0,1]$ the following subordination identities hold,
    \[
        \Ca f (s) = \alpha \int_0^\infty (1-e^{-t})^{\alpha-1}T(t)f(s) \, dt, \quad s \in [0,1],\, 1 < p \leq \infty,
    \]
    and
    \[
        \Ca^* f (s) = \alpha \int_0^\infty (1-e^{-t})^{\alpha-1}S(t)f(s) \, dt, \quad s \in [0,1],\, 1 \leq p < \infty.
    \]
\end{theorem}
\begin{proof}
    Let $\Real \alpha > 0$ and $f \in L^p[0,1]$. Then, the change of variable $u = e^{-t}s$ yields,
    \[
        \Ca f (s) = \frac{\alpha}{s^\alpha} \int_0^s (s-u)^{\alpha-1} f(u) \, du = \alpha \int_0^\infty (1-e^{-t})^{\alpha-1}T(t)f(s) \, dt,
    \]
    and, the change of variable $u = e^{t}s$ gives,
    \[
        \Ca^* f (s) = \alpha \int_s^1 \frac{(u-s)^{\alpha-1}}{u^\alpha} f(u) \, du = \alpha \int_0^\infty (1-e^{-t})^{\alpha-1}S(t)f(s) \, dt.
    \]
    Using the boundedness of the operators $T(t)$ and $S(t)$ we get
    \[
        \norm{\Ca f}_{p} \leq \abs{\alpha} \norm{f}_p \int_0^\infty (1-e^{-t})^{\Real \alpha-1} e^{-t / p'} \, dt = \abs{\alpha} \norm{f}_p \frac{\Gamma(\Real \alpha) \Gamma(1 / p')}{\Gamma(\Real \alpha + 1/p')}, \quad 1 < p \leq \infty,
    \]
    and
    \[
        \norm{\Ca^* f}_{p} \leq \abs{\alpha} \norm{f}_p \int_0^\infty (1-e^{-t})^{\Real \alpha-1} e^{-t / p} \, dt = \abs{\alpha} \norm{f}_p \frac{\Gamma(\Real \alpha) \Gamma(1 / p)}{\Gamma(\Real \alpha + 1/p)}, \quad 1 \leq p < \infty. \qedhere
    \]
\end{proof}

The following remark outlines the duality between the generalized Cesàro-Hardy operators. The proof is straightforward and is therefore omitted.

\begin{remark}
    \label{remark:duality_cesaro}
    Let $\Real \alpha > 0$. For $1 < p \leq \infty$, $\Ca$ on $L^p[0,1]$ is the dual operator of $\Ca^*$ on $L^{p'}[0,1]$.
\end{remark}

\begin{theorem}
    \label{th:spectrum_C*}
    Let $\Real \alpha > 0$ and $1 \leq p < \infty$. Then the operator $\Ca^* : L^p[0,1] \rightarrow L^p[0,1]$ satisfies
    \begin{enumerate}[(i)]
        \setlength\itemsep{0.5em}
        \item $ \displaystyle \sigma(\Ca^*) =  \left\{ \frac{\Gamma(\alpha+1)\Gamma(z+\frac{1}{p})}{\Gamma(\alpha+ z + \frac{1}{p})} \, : \, z \in \overline{\C_+}\right\} \cup \{0\}.$
        \item $\displaystyle \sigma_{point}(\Ca^*) = \emptyset.$
        \item $\displaystyle \left\{ \frac{\Gamma(\alpha+1)\Gamma(z+\frac{1}{p})}{\Gamma(\alpha+ z + \frac{1}{p})} \, : \, z \in i\R \right\} \cup \{0\} \subset \sigma_{ap}(\Ca^*)$.
        \item For $p=1$, $\displaystyle  \sigma_{r}(\Ca^*) = \left\{ \frac{\Gamma(\alpha+1)\Gamma(z+1)}{\Gamma(\alpha+ z + 1)} \, : \, z \in \overline{\C_+}\right\}$, \\ \phantom{a} \\ and for $1 < p < \infty$, $ \displaystyle \sigma_{r}(\Ca^*) =  \left\{ \frac{\Gamma(\alpha+1)\Gamma(z+\frac{1}{p})}{\Gamma(\alpha+ z + \frac{1}{p})} \, : \, z \in \C_+\right\}$.
        \item $\displaystyle \sigma_{ess}(\Ca^*) = \left\{ \frac{\Gamma(\alpha+1)\Gamma(z+\frac{1}{p})}{\Gamma(\alpha+ z + \frac{1}{p})} \, : \, z \in i\R \right\} \cup \{0\}.$
    \end{enumerate}
\end{theorem}
\begin{proof}

    For $1 \leq p < \infty$, the family $(e^{t/p} S(t))_{t \geq 0}$ forms a uniformly bounded $C_0$-semigroup on $L^p[0,1]$ with generator $(B + 1/p, D(B))$ (see \Cref{th:semigroups}). On the other hand, let $\tilde{\mu}$ be the finite Borel measure on $[0,\infty)$ and absolutely continuous with respect to the Lebesgue measure, defined by $d\tilde{\mu}(t) = \alpha e^{-t/p} (1-e^{-t})^{\alpha-1} dt$. Notice that the Laplace transform of the measure $\tilde{\mu}$ is,
    \[
        \LL(\tilde{\mu})(z):= \int_{[0,\infty)} e^{-zt} \mu(dt) = \int_0^\infty \alpha e^{-t(z+1/p)} (1-e^{-t})^{\alpha-1} dt = \frac{\Gamma(\alpha + 1)\Gamma(z+ 1/p)}{\Gamma(\alpha + z + 1/p)},
    \]
    for $\Real z \geq 0$ and $\LL(\tilde{\mu})$ is a bounded holomorphic function on $S_{\frac{\pi}{2}}$ and continuous on $\overline{S_{\frac{\pi}{2}}}$. Moreover, the function $\LL(\tilde{\mu})$ can be extended in a holomorphic way to $S_{\varphi}$, with $\varphi \in \left(\frac{\pi}{2}, \pi \right)$. This is because $\Gamma$ is a meromorphic function with poles in $\{0, -1,-2, \dots\}$ and $1/\Gamma$ is entire, so the poles of $\LL(\tilde{\mu})$ are at $z \in \{ -1/p, -1 - 1/p, -2-1/p, \dots\}$. Even more, we can also extend it in a holomorphic way to the shifted sector $S_\varphi - 1/p = \{\lambda - 1/ p \, : \, \lambda \in S_\varphi \}$.

    In addition, as $\LL(\tilde{\mu})(0) = \frac{\Gamma(\alpha + 1)\Gamma(1/p)}{\Gamma(\alpha + 1/p)}$ and is holomorphic in a neighborhood of $0,$ we have that $\LL(\tilde{\mu})$ has a finite polynomial limit at $0$, for a definition of polynomial limit see \cite[Page 27]{Haase2006}. Also, using the fact that 
    \begin{equation}
        \label{eq:quotient_gammas}
        \frac{\Gamma(z+\alpha)}{\Gamma(z)} = z^\alpha(1+ \frac{\alpha(\alpha-1)}{2z} + O(\abs{z}^{-2})), \quad \abs{z} \to \infty,
    \end{equation}
    whenever $z \neq 0, -1, -2, \dots$ and $z \neq -\alpha, -\alpha-1, \dots$, see \cite{ErdelyiTricomi1951}. We arrive to, $\LL(\tilde{\mu})(z) = O(\abs{z}^{-\alpha})$ as $|z|\to\infty,$ so $\lim_{\abs{z} \to \infty}\LL(\tilde{\mu})(z) = 0$ with $z \in S_\varphi$ and $\LL(\tilde{\mu})$ has finite polynomial limit at $\infty$. Using \cite[Lemma 2.2.3]{Haase2006} we have that $\LL(\tilde{\mu}) \in \mathcal{E}(S_\varphi)$, where $\mathcal{E}(S_\varphi)$ is the extended Dunford-Riesz class. Since $-B - 1/p$ is a sectorial operator of angle $\frac{\pi}{2}$, by the spectral mapping theorem \cite[Theorem 2.7.8]{Haase2006} and using $\Ca^* =\LL(\tilde{\mu})(-B-1/p)$, we have
    \[
        \sigma(\Ca^*) = \LL(\tilde{\mu})(\sigma(-B-1/p)) \cup \{0\}= \left\{ \frac{\Gamma(\alpha+1)\Gamma(z+\frac{1}{p})}{\Gamma(\alpha+ z + \frac{1}{p})} \, : \, z \in \overline{\C_+}\right\} \cup \{0\}.
    \]

Item (iv) follows from the fact that the dual operator of $\Ca^*$ defined on $L^p[0,1]$ with $1 \leq p <\infty$ is $\Ca$ defined on $L^{p'}[0,1]$. For a closed densely defined operator, its residual operator coincides with the point spectrum of its dual, see \cite[Proposition IV.1.12]{EngelNagel2000} and thus item (ii) of \Cref{th:spectrum_C} below finishes the proof of this item.

    Item (v) follows from the spectral mapping theorem for the essential spectrum that appears in \cite[Theorem 5.4]{Oliva-Maza2023},
    \[
        \sigma_{ess}(\Ca^*) = \left\{ \frac{\Gamma(\alpha+1)\Gamma(z+\frac{1}{p})}{\Gamma(\alpha+ z + \frac{1}{p})} \, : \, z \in i\R \right\} \cup \{0\},
    \]
    and by \eqref{eq:quotient_gammas} and \cite[Theorem 5.6]{Oliva-Maza2023} we have that,
    \[
        \sigma_{point}(\Ca^*) =\LL(\tilde{\mu})(\sigma_{point}(-B-1/p)) = \emptyset.
    \]
    In \cite[Theorem 5.4 and Theorem 5.6]{Oliva-Maza2023} the author uses the notion of quasi-regularity; notice that the property of having polynomial limit is stronger, so we are under the assumptions of those results.

    Lastly, item (iii) follows from the fact that the boundary of the spectrum is always contained in the approximate spectrum and the fact that $\sigma_{ap}(\Ca^*) \cup \sigma_r (\Ca^*)= \sigma(\Ca^*)$.
\end{proof}

\begin{theorem}
    \label{th:spectrum_C}
    Let $\Real \alpha > 0$ and $1 < p \leq \infty$. Then the operator $\Ca: L^p[0,1] \rightarrow L^p[0,1]$ satisfies
    \begin{enumerate}[(i)]
        \setlength\itemsep{0.5em}
        \item $ \displaystyle \sigma(\Ca) =  \left\{ \frac{\Gamma(\alpha+1)\Gamma(z+\frac{1}{p'})}{\Gamma(\alpha+ z + \frac{1}{p'})} \, : \, z \in \overline{\C_+}\right\} \cup \{0\}.$
        \item  For $1 < p < \infty$, $ \displaystyle \sigma_{point}(\Ca) =  \left\{ \frac{\Gamma(\alpha+1)\Gamma(z+\frac{1}{p'})}{\Gamma(\alpha+ z + \frac{1}{p'})} \, : \, z \in \C_+\right\},$ \\ \phantom{a} \\ and for $p = \infty$, $\displaystyle  \sigma_{point}(\Ca) = \left\{ \frac{\Gamma(\alpha+1)\Gamma(z+1)}{\Gamma(\alpha+ z + 1)} \, : \, z \in \overline{\C_+}\right\}$.
        \item $ \displaystyle \sigma_{ap}(\Ca) =  \left\{ \frac{\Gamma(\alpha+1)\Gamma(z+\frac{1}{p'})}{\Gamma(\alpha+ z + \frac{1}{p'})} \, : \, z \in \overline{\C_+}\right\} \cup \{0\}.$
        \item For $1 < p < \infty$, $ \displaystyle \sigma_{r}(\Ca) = \emptyset$.
        \item $\displaystyle \sigma_{ess}(\Ca) = \left\{ \frac{\Gamma(\alpha+1)\Gamma(z+\frac{1}{p'})}{\Gamma(\alpha+ z + \frac{1}{p'})} \, : \, z \in i\R \right\} \cup \{0\}$.
    \end{enumerate}
\end{theorem}
\begin{proof}
    Item (i) and (v) follow from the duality property stated in \Cref{remark:duality_cesaro} and the duality of the spectrum and essential spectrum, see \cite[Th. IX.1.1]{EdmundsEvans1987}. 

    When $1< p < \infty$ the result for the point spectrum follows the same ideas as the proof of the point spectrum for the dual Cesàro-Hardy operator. For the $p=\infty$ case notice that by the result for the point spectrum for the family $(T(t))_{t \geq 0},$ we have that $g_\lambda (s) = \frac{1}{s^{\lambda + 1}}$ is an eigenfunction with eigenvalue $e^{t \lambda}$ with $\Real \lambda \leq -1$, therefore,
    \[
        \begin{aligned}
            \Ca g_\lambda  = \alpha \int_0^\infty (1-e^{-t})^{\alpha-1} T(t) g_\lambda \, dt = g_\lambda\alpha \int_0^\infty (1-e^{-t})^{\alpha-1} e^{t \lambda} \, dt  = \frac{\Gamma(\alpha+1)\Gamma(-\lambda)}{\Gamma(\alpha - \lambda)}  g_\lambda,
        \end{aligned}
    \]
    with $\Real \lambda \leq -1$, therefore 
    \[
        \left\{ \frac{\Gamma(\alpha+1)\Gamma(z+1)}{\Gamma(\alpha+ z + 1)} \, : \, z \in \overline{\C_+}\right\} \subseteq \sigma_{point}(\Ca).
    \]

    Moreover, $0 \notin \sigma_{\text{point}}(\Ca)$ by a theorem of E.C. Titchmarsh \cite[Theorem VII]{Titchmarsh1926}, which states that if $f, g \in L^1[0,1]$ satisfy  
    \[
        \int_0^s f(u-s) g(u) \, du = 0 \quad \text{for almost every } s \in (0,1),
    \]
    and $f \neq 0$ almost everywhere on $[0,1]$, then $g$ must be zero almost everywhere on $[0,1]$. As $\Ca (s) = \frac{\alpha}{s^\alpha} \int_0^{s} (s-u)^{\alpha-1}f(u) \, du$, the last theorem implies that if we have $\Ca f = 0$ then $f=0$ so $\Ca$ is injective. As $\sigma_{point}(\Ca) \subseteq \sigma(\Ca)$ the result follows.
    
    Item (iv) follows the same argument as in the proof of \Cref{th:spectrum_C*}, although in this case the proof doesn't apply for the case  $p=\infty$.
    
    Lastly, for item (iii), notice that when $1<p<\infty$ by $\sigma_{ap}(\Ca) \cup \sigma_{r} (\Ca)= \sigma(\Ca),$ item (i) and (iv), the result follows. For $p = \infty$, notice that $0 \in \sigma_{ap}(\Ca)$. Indeed, consider the sequence $f_n(s) = s^n$, $s \in (0,1)$ and $n \in \N$. Thus,
    \[
        \abs{\Ca f_n(s)} \leq \abs{\alpha} s^n \int_0^\infty (1-e^{-t})^{\Real \alpha-1} e^{-(n+1) t} \, dt \to 0, \quad \text{as } n \to \infty
    \]
    by the dominated convergence theorem. Using the characterization of the approximate point spectrum via sequences (see \cite[Lemma IV.1.9]{EngelNagel2000}) we have that $0 \in \sigma_{ap}(\Ca)$. Also, by definition, $\sigma_{point}(\Ca) \subseteq \sigma_{ap}(\Ca),$ so the result follows.
\end{proof}
\begin{corollary}
    \label{cor:cesaro_norm}
    Let $\alpha > 0$. Then:
    \begin{enumerate}[(i)]
        \item The operator $\Ca$ is a bounded operator on $L^p[0,1]$, for $1 < p \leq \infty$ with norm $$\norm{\Ca} = \frac{\Gamma(\alpha+1) \Gamma(\frac{1}{p'})}{\Gamma(\alpha+\frac{1}{p'})}.$$
        \item The operator $\Ca^*$ is a bounded operator on $L^p[0,1]$, for $1 \leq p < \infty$ with norm $$\norm{\Ca^*} = \frac{\Gamma(\alpha+1) \Gamma(\frac{1}{p})}{\Gamma(\alpha+\frac{1}{p})}.$$ 
    \end{enumerate}
\end{corollary}
\begin{proof}
    The result follows from the preliminary bounds computed in \Cref{th:cesaro_boundedness} and the fact that $\frac{\Gamma(\alpha+1) \Gamma(\frac{1}{p'})}{\Gamma(\alpha+\frac{1}{p'})} \in \sigma(\Ca)$ and  $\frac{\Gamma(\alpha+1) \Gamma(\frac{1}{p})}{\Gamma(\alpha+\frac{1}{p})} \in \sigma(\Ca^*)$ by \Cref{th:spectrum_C} and \Cref{th:spectrum_C*}, providing lower bounds for the spectral radii and therefore for the operator norms.
\end{proof}
\begin{remark}
    \label{rem5.5}
    Note that in particular the classical Cesàro-Hardy operator $\CC:=\CC_1$ is the inverse of the differential operator $-A$ (see \Cref{eq:integral_Lambda_operators} and  \Cref{prop:integral_resolvent}), that is, $$\CC f(s)=\frac{1}{s}\int_0^s f(u)\,du=\int_0^{\infty} (T(t)f)(s)\,dt=(-A)^{-1}f(s),\quad s\in (0,1),\, f\in L^p[0,1].$$
\end{remark}

\begin{example}[Spectral pictures for $p=2$]
    \label{ex:spectral_pictures}
    When $p=2$, \Cref{th:spectrum_C*,th:spectrum_C} give the same spectrum and essential spectrum for $\Ca$ and $\Ca^*$. These sets are illustrated in \Cref{fig:spectral_pictures} for $\alpha=1,3/2,1/2$, and $1+i$. When $\alpha=1$, the quotient appearing in the spectral formulas reduces to $(z+1/2)^{-1}$, so the spectrum is the closed disk with centre $1$ and radius $1$. For real $\alpha$ the pictures are symmetric with respect to the real axis, whereas this symmetry is lost for the complex order $\alpha=1+i$.
\end{example}

\begin{figure}[!htbp]
    \centering
    \begin{minipage}{0.48\textwidth}
        \centering
        \includegraphics{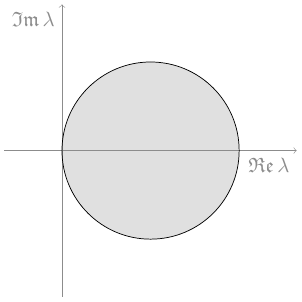}

        \small $\alpha=1$
    \end{minipage}\hfill
    \begin{minipage}{0.48\textwidth}
        \centering
        \includegraphics{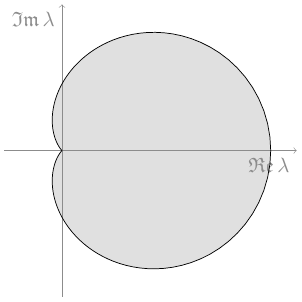}

        \small $\alpha=3/2$
    \end{minipage}

    \medskip

    \begin{minipage}{0.48\textwidth}
        \centering
        \includegraphics{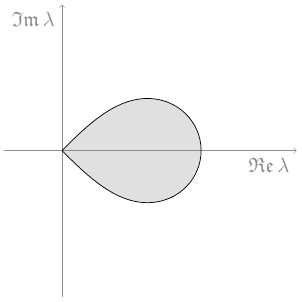}

        \small $\alpha=1/2$
    \end{minipage}\hfill
    \begin{minipage}{0.48\textwidth}
        \centering
        \includegraphics{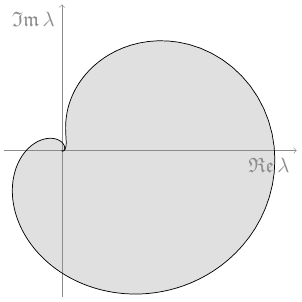}

        \small $\alpha=1+i$
    \end{minipage}
    \caption{Spectral pictures for $\Ca$ and $\Ca^*$ on $L^2[0,1]$. The spectrum is shaded light grey and the essential spectrum is drawn as a thin black curve. In these four cases the essential spectrum coincides with the boundary of the spectrum.}
    \label{fig:spectral_pictures}
\end{figure}

\begin{example}[Essential spectrum in the interior]
    \label{ex:essential_spectrum_interior}
    The essential spectrum need not coincide with the boundary of the spectrum. Consider $\mathcal C_3$ on $L^3[0,1]$. In this case, the quotient in \Cref{th:spectrum_C} reduces to
    \[
        \frac{\Gamma(4)\Gamma(z+\frac23)}{\Gamma(z+\frac{11}{3})}
        =\frac{6}{(z+\frac23)(z+\frac53)(z+\frac83)}.
    \]
    Taking $z=it$, with $t\in\R$, gives the essential spectrum together with $0$. This curve crosses itself at $-81/455$, corresponding to $t=\pm\sqrt{22/3}$. The part with $\abs{t}\leq\sqrt{22/3}$ forms the outer boundary, while the two remaining branches and their common limit $0$ form a small loop whose points, except for the crossing point, lie in the interior of the spectrum; see \Cref{fig:essential_spectrum_interior}. Consequently,
    \[
        \partial\sigma(\mathcal C_3)
        \subsetneq \sigma_{ess}(\mathcal C_3).
    \]
\end{example}

\begin{figure}[!htbp]
    \centering
    \includegraphics{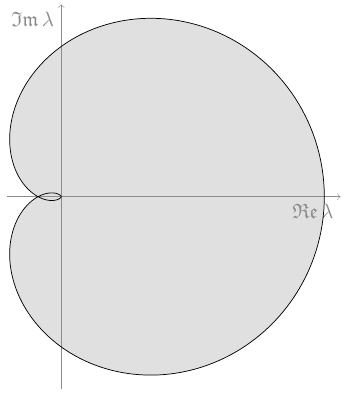}
    \caption{The spectrum of $\mathcal C_3$ on $L^3[0,1]$ is shaded light grey. The complete thin black curve is the essential spectrum.}
    \label{fig:essential_spectrum_interior}
\end{figure}

\section{Intertwining relations with discrete generalized Cesàro operators on \texorpdfstring{$\ell^p$}{lp}}
\label{sec:discrete}

In this section we consider the usual sequence Lebesgue space  $\ell^p$ with the norm
$$
\left(\sum_{n=0}^\infty \vert a(n)\vert^p\right)^{\frac{1}{p}},
$$
for $1\le p<\infty,$ and $\sup_n\vert a(n)\vert<\infty$ for $p=\infty.$ Let us denote by $\mathscr{C}_{\alpha}$ and $\mathscr{C}_{\alpha}^*$ the generalized discrete Ces\`{a}ro operator for $\Real\alpha>0$ and its dual operator, given by
\begin{eqnarray*}
\mathscr{C}_{\alpha}a(n)&:=&\frac{1}{k^{\alpha+1}(n)}\sum_{j=0}^nk^{\alpha}(n-j)a(j), \cr \mathscr{C}^*_{\alpha}a(n)&:=&\sum_{j=n}^{\infty}\frac{1}{k^{\alpha+1}(j)}k^{\alpha}(j-n)a(j),
\end{eqnarray*}
for $a\in \ell^p,$ and $k^\alpha(n):=\displaystyle{\frac{\Gamma(n+\alpha)}{\Gamma(\alpha)\Gamma(n+1)}}$ which are the classical Cesàro numbers for $n \in \N_0$. Note that $\mathscr{C}_1=C_0.$

The purpose of this section is to exhibit intertwining relations between the continuous and discrete models, first at the level of the underlying semigroups and then, through the subordination formulas, at the level of the generalized Cesàro operators. These identities provide a structural link between the two settings, but they should not be interpreted as similarity relations: the intertwining maps introduced below are not invertible. In particular, we do not use them to transfer spectral properties or invariant-subspace information from one setting to the other.

In the next theorem we introduce the operators that implement these intertwining relations. The proof is straightforward and is left to the reader.
\begin{theorem} Take $1\le p \le \infty$.
\begin{itemize}
\item[(i)] The operator $V: L^p[0,1]\to \ell^p$ given by
$$
V(f)(n):= \int_0^1{\frac{t^n}{n!}}e^{-t}f(t)dt, \qquad n\in \N, \quad f\in L^p[0,1],
$$
is well-defined, linear and $\Vert V\Vert \le 1$.
\item[(ii)] The operator $V^*: \ell^p\to L^p[0,1]$ given by
$$
V^*(a)(t):=\sum_{n=0}^\infty {\frac{t^n}{n!}}e^{-t}a(n), \qquad t\in [0,1], \quad a\in \ell^p,
$$
is also well-defined, linear and $\Vert V\Vert \le 1$. In particular, $V^*$ on $\ell^p$ is the dual operator of $V$ on $L^{p'}[0,1]$ for $1<p\leq \infty.$ 
\end{itemize}
\end{theorem}

\begin{remark}Note that both operators are not invertible for $1\le p\le \infty$. In the case that $V$ was invertible, there will exist $f\in L^p[0,1]$ such that $V(f)=\delta_0$, where $\delta_0(n):=1$ if $n=0$ and  $\delta_0(n):=0$ for $n>0$. Then 
$$
\sum_{n=0}^{\infty} Vf(n)z^n=1= \int_0^1 f(t)e^{-t(1-z)}dt, \qquad z\in\C,\, |z|<1.
$$
This implies that the function $f$ should be the delta Dirac distribution, which is a contradiction. 

Now we suppose that $V^*$ is invertible for $1\le p\le \infty$. Then there exists $a\in \ell^p$ such that $V^*(a)=\chi_{[0,1]}$, i.e.,
$$e^{t}=\sum_{n=0}^\infty a(n){\frac{t^n}{n!}}, \qquad t\in [0,1].
$$
Then $a(n)=1 $ for $n\ge 0$ and we obtain a contradiction $(a\in \ell^p)$ for $1\le p<\infty$. For $p=\infty$, and  $V^*$ is invertible, there exists $b\in \ell^\infty$ such that $V^*(b)(t)=t$, i.e.,
$$te^{t}=\sum_{n=0}^\infty b(n){\frac{t^n}{n!}}, \qquad t\in [0,1],
$$
and we obtain a contradiction with $b\in\ell^\infty$.

\end{remark}

Now we consider the one-parameter families of operators $(\mathscr{T}(t))_{t\geq 0}$ and $(\mathscr{S}(t))_{t\geq 0}$ acting on  $\ell^p$ with $1\le p\le \infty$, where
\begin{eqnarray}\label{semis}
\mathscr{T}(t)a(n)&:=&\sum_{j=0}^n \binom{n}{j} e^{-tj}(1-e^{-t})^{n-j}a(j),\cr \mathscr{S}(t)a(n)&:=&e^{-tn}\sum_{j=n}^{\infty} \binom{j}{n} (1-e^{-t})^{j-n}a(j),
\end{eqnarray}
for $n\in \N_0$ and $a\in \ell^p$. In fact, both families  $(\mathscr{T}(t))_{t\geq 0}$ and $(\mathscr{S}(t))_{t\geq 0}$ are dual $C_0$-semigroups in $\ell^p$, $\Vert \mathscr{T}(t)\Vert\le e^{\frac{t}{p}} $ and  $\Vert \mathscr{S}(t)\Vert\le e^{t(1-{\frac{1}{p}})} $ for $1\le p\le \infty$. As we have explained in the introduction, these semigroups were introduced in \cite{AbadiasMiana2018} to study spectral properties of $C_0$ and its generalizations $\mathscr{C}_{\alpha}$ and $\mathscr{C}_{\alpha}^*$ on $\ell^p$ (and also on other sequence spaces).

In the following theorem, we present that the one-parameter families $(T(t))_{t\geq0}$, $(S(t))_{t\geq 0}$,  $(\mathscr{T}(t))_{t\geq 0}$ and $(\mathscr{S}(t))_{t\geq 0}$ intertwines with $V$ and $V^*$.
\begin{theorem}
    \label{intert}
    Let $1\le p\le \infty,$ then
    \begin{itemize}
        \item[(i)] $\mathscr{S}(t)\circ V= e^t(V\circ S(t))$ for $t>0$.
        \item[(ii)] $ V^*\circ\mathscr{T}(t)= e^t( T(t)\circ V^*)$ for $t>0$.
    \end{itemize}
\end{theorem}
\begin{proof} (i) Take $t>0$ and $f\in L^p[0,1]$ with $1\le p\le\infty$. For $n\in \N_0$, one obtains that 
\begin{eqnarray*}
(\mathscr{S}(t)\circ V)(f)(n)&=&e^{-tn}\displaystyle\sum_{j=n}^{\infty}\binom{j}{n}(1-e^{-t})^{j-n} \int_0^1\frac{s^j}{j!}e^{-s}f(s)\,ds\cr
&=&e^{-tn}\int_0^1 e^{-s}f(s)\displaystyle\sum_{j=n}^{\infty}\frac{1}{n!(j-n)!}(1-e^{-t})^{j-n}s^j\,ds\cr
&=&e^{-tn}\int_0^1 e^{-s}f(s)\frac{s^n}{n!}\displaystyle\sum_{m=0}^{\infty}\frac{(s(1-e^{-t}))^m}{m!}\,ds\cr
&=&e^{-tn}\int_0^1 e^{-s}f(s)\frac{s^n}{n!}e^{s(1-e^{-t})}\,ds \cr
&=&e^{t}\int_0^{e^{-t}}f(e^tu)\frac{u^n}{n!}e^{-u}\,du=  e^t(V\circ S(t))(f)(n).
\end{eqnarray*}
    (ii) Now take  $t>0$ and $a\in \ell^p$ with $1\le p\le\infty$.  For $0<s<1$, we get

\begin{eqnarray*}
(V^*\circ\mathscr{T}(t))(a)(s)&=&\sum_{n=0}^\infty e^{-s}\frac{s^n}{n!}\displaystyle\sum_{j=0}^{n}\binom{n}{j}e^{-tj}(1-e^{-t})^{n-j}a(j) \cr
&=&e^{-s}\sum_{j=0}^\infty \frac{e^{-tj}}{j!}a(j)\displaystyle\sum_{n=j}^{\infty}\frac{ s^n(1-e^{-t})^{n-j} }{(n-j)!}\cr
&=&e^{-s}\sum_{j=0}^\infty \frac{(se^{-t})^j}{j!}a(j)\displaystyle\sum_{m=0}^{\infty}\frac{(s(1-e^{-t}))^m}{m!}\,ds\cr
&=&\sum_{j=0}^\infty \frac{(se^{-t})^j}{j!}e^{-se^{-t}}a(j)= e^t( T(t)\circ V^*)(a)(s),
\end{eqnarray*}
and we conclude the proof.
\end{proof}

\begin{remark} In the case that $1\le p <\infty$, we may prove Theorem \ref{intert} (ii) by duality. For $t>0$ we get 
$$
 V^*\circ\mathscr{T}(t)= (\mathscr{S}(t)\circ V)^*= e^t(V\circ S(t))^*=  e^t (T(t)\circ V^*).
$$
\end{remark}
The following theorem presents the connection between  ${\mathcal{C}}_{\alpha}$, ${\mathcal{C}}_{\alpha}^*$,  $\mathscr{C}_{\alpha}$  and $\mathscr{C}_{\alpha}^*$.
Remind that for $a\in \ell^p$ and $n\in\N_0$, we have that
\begin{eqnarray*}
{\mathscr{C}}_{\alpha}a(n)&=&\displaystyle\alpha\int_0^{\infty}(1-e^{-t})^{\alpha-1}e^{-t}\mathscr{T}(t)a(n)\,dt,\quad \, 1<p\leq\infty,\cr
{\mathscr{C}}^*_{\alpha}a(n)&=&\alpha\int_0^{\infty}(1-e^{-t})^{\alpha-1}e^{-t}\mathscr{S}(t)a(n)\,dt,\quad \, 1\leq p <\infty,
\end{eqnarray*}
see \cite[Theorem 7.1] {AbadiasMiana2018}.

\begin{theorem} Take  $\Real\alpha >0.$ Then
\begin{itemize} 
\item[(i)] $V\circ  {\mathcal{C}}_{\alpha}^*=\mathscr{C}^*_{\alpha}\circ V$ for  $1\leq p <\infty$.
\item[(ii)] $V^*\circ\mathscr{C}_{\alpha}=  {\mathcal{C}}_{\alpha}\circ V^*$ for  $1<p\leq\infty$.
\end{itemize}
    

\end{theorem}
\begin{proof} (i) Take $ f\in L^p[0,1] $ for $1\leq p <\infty$.  By Theorem  \ref{th:cesaro_boundedness} and Theorem \ref{intert}(i), we obtain that
\begin{eqnarray*}
(V \circ  {\mathcal{C}}_{\alpha}^*f)(n) &=&\alpha \int_0^\infty (1-e^{-t})^{\alpha-1}(V\circ S(t)f)(n)ds \cr
&=&\alpha \int_0^\infty (1-e^{-t})^{\alpha-1}e^{-t}(\mathscr{S}(t)\circ Vf)(n)ds=(\mathscr{C}^*_{\alpha}\circ Vf)(n), \cr
\end{eqnarray*}
for $n\in \N_0$. Similarly, it is proven the item (ii). 
\end{proof}

We emphasize that the preceding identities are intertwining relations, not similarity relations. As observed above, neither $V$ nor $V^*$ is invertible. Consequently, these identities do not by themselves imply equality of spectra, fine spectra, or invariant-subspace lattices between the continuous and discrete operators. The spectral equalities recorded in the next remark follow instead by comparing the independent spectral formulas proved here with those obtained in \cite[Theorem 7.6]{AbadiasMiana2018}.

\begin{remark} Take $\Real\alpha>0$, one can observe that 
$$\sigma(\mathcal{C_{\alpha}})=\sigma(\mathscr{C}_{\alpha})= \left\{ \frac{\Gamma(\alpha+1)\Gamma(z+\frac{1}{p'})}{\Gamma(\alpha+ z + \frac{1}{p'})} \, : \, z \in \overline{\C_+}\right\} \cup \{0\},$$ for $1<p\leq\infty, $
 see Theorem \ref{th:spectrum_C} and \cite[Theorem 7.6]{AbadiasMiana2018}, and 
 $$\sigma(\mathcal{C^*_{\alpha}})=\sigma(\mathscr{C}^*_{\alpha})= \left\{ \frac{\Gamma(\alpha+1)\Gamma(z+\frac{1}{p})}{\Gamma(\alpha+ z + \frac{1}{p})} \, : \, z \in \overline{\C_+}\right\} \cup \{0\},$$ for $1\leq p <\infty,$   see Theorem \ref{th:spectrum_C*}  and \cite[Theorem 7.6]{AbadiasMiana2018}.

\end{remark}

\section{Universality and closed invariant subspaces for the semigroup \texorpdfstring{$T(t)$}{T(t)}}
\label{sec:universality_lattice}

The semigroup $(T(t))_{t\geq 0}$ has played a central role throughout the paper in the study of the Cesàro-Hardy operators. Thus, beyond its spectral properties, it is natural to investigate its invariant subspaces and their relation with those of the Cesàro-Hardy operator. This point of view is also motivated by recent works on Cesàro and Hardy operators in which semigroup methods are used to study their invariant subspaces; see, for instance, \cite{Gallardo-GutierrezPartington2024,Gallardo-GutierrezPartingtonRoss2025,Gallardo-GutierrezPartingtonRoss2026}. In the classical case $\mathcal C=\mathcal C_1$, this connection becomes particularly explicit: we will show that the closed subspaces invariant under $T(t)$ for every $t>0$ coincide with the closed invariant subspaces of $\mathcal C$. On the other hand, the spectral properties obtained in the previous sections also lead to the universality of certain translations of $T(t)$, providing a connection with the Invariant Subspace Problem. These two observations motivate the study carried out in this final section.

The Invariant Subspace Problem (ISP) is one of the most important open problems in Operator Theory on Hilbert Spaces. This problem is about finding if every linear bounded operator on an infinite-dimensional separable complex Hilbert space has a non-trivial closed invariant subspace. In 1960, G. C. Rota \cite{Rota1959,Rota1960} presented the concept of an operator whose lattice of invariant subspaces possesses a sufficiently rich structure to model any Hilbert space operator, and demonstrated the existence of such operators. These are the universal operators. 

\begin{definition}
    Let $H$ be an infinite-dimensional separable complex Hilbert space and let $\BB(H)$ denote the algebra of bounded linear operators on $H$. An operator $U \in \BB(H)$ is called universal if for every $T \in \BB(H)$, there exist $\lambda \in \C$, $M \in \lat U$ (the space of closed invariant subspaces for $U$), and an isomorphism $\Theta : H \to M$ such that $\Theta T = \lambda U \mid_M \Theta$. In other words, $T$ is similar to $\lambda U \mid_M$.
\end{definition}

Now, the similarity $\Theta$ maps invariant subspaces of $T$ to those of $U \mid_M$. Suppose $U$ is a universal operator for a separable, infinite-dimensional complex Hilbert space $H$. Then every bounded operator on $H$ has an invariant subspace if and only if every infinite-dimensional subspace $M$ of $H$ that is invariant under $U$ has a non-zero proper subspace $M_0$ that is also invariant under $U$. Thus, we have the following known characterization.

\begin{theorem}
    \label{th:equiv_isp}
    Let $H$ be a separable Hilbert space and $U \in \BB(H)$ be a universal operator. Then the following conditions are equivalent: 
    \begin{enumerate}[(i)]
        \item Every operator $T \in \BB(H)$ has a non-trivial closed invariant subspace.
        \item Every minimal non-trivial closed invariant subspace of $U$ is one-dimensional.
    \end{enumerate}
\end{theorem}

Following the previous theorem, we say that a non-trivial closed invariant subspace for $U$ is minimal if it contains no proper non-trivial closed invariant subspaces for $U$. Consequently, understanding the Invariant Subspace Problem for Hilbert spaces reduces to understanding the invariant subspaces of the single operator $U,$ because all infinite-dimensional separable Hilbert spaces are isomorphic.

In 1969, S. R. Caradus \cite{Caradus1969} established a systematic tool to produce  universal operators on Hilbert spaces.

\begin{theorem}
    Let $H$ be a separable Hilbert space and $T \in \BB(H)$ such that:
    \begin{enumerate}[(i)]
        \item $\dim \ker T = \infty$,
        \item and $T$ is surjective,
    \end{enumerate}
    then $T$ is universal for $H$.
\end{theorem}

In the following we see that the properties of the spectrum of $T(t)$ studied in \Cref{sec:spectra_semigroups} give as a consequence the universality of translations of $T(t)$ on $L^2[0,1]$. Additionally, we will present several results concerning the minimal subspaces of $T(t)$. For convenience, we denote $\S := \{ \lambda \in \C \, : \, \Real \lambda < -1/2 \}.$ 

\begin{theorem}
    \label{th:T_is_universal}
    Let $t>0$ and $\mu \in B(0,e^{-t/2}),$ then the operator $\mu - T(t)$ is universal on $L^2[0,1]$. In particular, if every minimal non-trivial closed invariant subspace of $T(t)$ is one-dimensional for some $t>0$, then the Invariant Subspace Problem has a positive answer.
\end{theorem}
\begin{proof}
    On one hand, as we have seen in the proof of \Cref{th:fine_spectrum_T} we have that $\ker (\mu - T(t))$ is infinite-dimensional. In particular, if $\mu=e^{\lambda t}$ with $\lambda\in\S,$ \begin{equation*}\ker (\mu - T(t))=\overline{\operatorname{span} \{ g_{\lambda+\frac{2k\pi i}{t}}\,:\,k\in\Z \}}^{L^2[0,1]}.\end{equation*}
    
    On the other hand, as $\sigma_{r}(T(t)) = \emptyset$ (see item (iv) of \Cref{th:fine_spectrum_T}) we have that $\ran (\mu - T(t))$ is dense. Lastly, by the Closed Range Theorem, (see \cite[page 205]{Yosida1980} for a reference)  $\ran (\mu - T(t))$ is closed if and only $\ran (\mu - S(t))$ is closed. Since $\sigma_{point}(S(t)) = \emptyset$ and $\sigma_{ap}(S(t)) = \delta B(0,e^{-t/2})$ (see \Cref{th:fine_spectrum_S}) this means that for $\mu \in B(0,e^{-t/2})$, $\ran (\mu - S(t))$ is closed and thus $\ran (\mu - T(t))$ is closed as well. The application of the Caradus criterion stated before finishes the proof.

    Lastly, as the closed invariant subspaces of $\mu - T(t)$ and $T(t)$ are the same, \Cref{th:equiv_isp} finishes the proof, or merely taking $\mu=0.$
\end{proof}

\begin{remark}
The universality in \Cref{th:T_is_universal} can also be viewed through the classical backward translation model. Indeed, consider the unitary operator
\[
U:L^2(0,1)\longrightarrow L^2(0,\infty),\qquad
(Uf)(x)=e^{-x/2}f(e^{-x}).
\]
If $(\tau_t)_{t\geq0}$ denotes the backward translation semigroup on $L^2(0,\infty)$,
\[
(\tau_tg)(x)=g(x+t),
\]
then a direct computation gives
\[
UT(t)U^{-1}=e^{-t/2}\tau_t.
\]
Consequently,
\[
U(\mu-T(t))U^{-1}=e^{-t/2}\bigl(e^{t/2}\mu-\tau_t\bigr).
\]
The backward translation operator is the standard realization on $L^2(0,\infty)$ of the backward shift of infinite multiplicity; see \cite{Rota1960,CelariesChalendarPartington2019}. Hence, when $|\mu|<e^{-t/2}$, the universality of $\mu-T(t)$ also follows by applying the classical Caradus criterion \cite{Caradus1969} to $e^{t/2}\mu-\tau_t$. This gives an alternative interpretation of \Cref{th:T_is_universal} in terms of the backward translation model.
\end{remark}

As discussed in the introduction, our goal is to understand the eigenfunctions and closed invariant subspaces of the operator $T(t)$ for $t > 0$. The subspaces 
\[
    M_{r}:=M_{2,r} = \{ f \in L^2[0,1] \mid f(s) = 0 \text{ a.e. for } s \in (0,r) \}, \quad r \in (0,1),
\]
defined at the beginning of \Cref{sec:spectra_semigroups}, are invariant under $T(t)$. Functions in $M_{e^{-t}}$ are all eigenfunctions with eigenvalue $0$. Moreover, for each $r \in (0,1)$, $M_r \cap M_{e^{-t}} \neq \emptyset$, implying that every closed invariant subspace $M_r$ contains an eigenfunction. In particular, $T(t)$ is nilpotent on every closed invariant subspace $M_r$ with $r\in (0,1).$

Furthermore, for $\lambda \in \S$, the functions of the form $g_{\lambda+\frac{2k\pi i}{t}},$ $k\in\Z,$ are eigenfunctions corresponding to the eigenvalue $e^{t\lambda} \in B(0,e^{-t/2})$. Consequently, the subspace spanned by each such eigenfunction is also invariant under $T(t)$. In particular, from the full M\"untz–Szász theorem in $L^2[0, 1]$ (see \cite{Szasz1953}) one can deduce that if $\{s_n\}_{n\in\N}\subset \S,$ then the closed invariant subspaces $$\overline{\operatorname{span} \{ g_{\lambda}\,:\,\lambda\in \{s_n\}_{n\in\N} \}}^{L^2[0,1]}$$ of $T(t)$ are non-trivial if and only if $$\sum_{n=0}^{\infty} \frac{-(2\Real s_n+1)}{|s_n|^2}<\infty.$$

In \cite[Proposition 3.8]{AglerMcCarthy2023_Beurling} the authors also showed that for $\lambda\in\S,$ the $(m + 1)^{\text{st}}$ generalized eigenspace of $\CC$ ($=\CC_1$) with eigenvalue $-\frac{1}{\lambda}$ lies in the linear span of $\{ g_{\lambda}, g_{\lambda} \log, \ldots , g_{\lambda}\log^m \},$ that is, $$\ker (\CC+\frac{1}{\lambda})^{m+1}=\operatorname{span} \{ g_{\lambda}, g_{\lambda} \log, \ldots , g_{\lambda}\log^m \}.$$ Similarly, for $t>0$ given, one can get for $\lambda\in\S$ that the $(m + 1)^{\text{st}}$ generalized eigenspace of $T(t)$ with eigenvalue $e^{\lambda t}$ is the nontrivial closed invariant subspace \begin{equation*}\label{ker}\ker (e^{\lambda t} - T(t))^{m+1}=\overline{\operatorname{span} \{ g_{\lambda+\frac{2k\pi i}{t}}, g_{\lambda+\frac{2k\pi i}{t}} \log, \ldots , g_{\lambda+\frac{2k\pi i}{t}}\log^m \,:\,k\in\Z \}}^{L^2[0,1]},\end{equation*}
since in general $T(t)g_{\lambda}\log^m=e^{\lambda t} g_{\lambda}\sum_{j=0}^m (-t)^{m-j}\log^j.$

The natural next question is whether there are closed invariant subspaces of $T(t)$ beyond those already identified. Notably, the closed invariant subspaces identified so far are independent of $t$, prompting us to ask if there are any that depend on the value of $t$.

We now turn our attention to the relationship between the closed invariant subspaces of $T(t)$ and the Cesàro-Hardy operator defined on $L^2[0,1]$ by
\[
    \CC f(s) =  \frac{1}{s} \int_0^s f(u) \, du = \int_0^\infty T(t) f(s) \, dt, \quad s \in (0,1].
\]
Note that for a closed subspace $M$ invariant under $T(t)$ for all $t > 0$, $M$ is also invariant under $\CC$. In other words, this means that 
\begin{equation}
    \label{eq:Lat_T(t)_subset_C}
    \cap_{t>0} \lat T(t) \subseteq \lat \CC.
\end{equation}
The closed invariant subspaces of $\CC$ have been recently characterized in \cite[Theorem 1.4]{AglerMcCarthy2023_Beurling}, introducing the concept of monomial space, which we present in the following lines.

Recall that $\lambda \in \S$ if and only if $g_\lambda\in L^2[0,1]$. For $S$ a finite subset of $\S$ we let $\MM(S)$ denote the span in $L^2[0,1]$ of the monomials $g_{\lambda}$ whose exponents lie in $S$, i.e.,
\[
    \MM(S) = \left\{ \sum_{\lambda\in S} a(\lambda) g_\lambda \, : \, a(\lambda)\in \C \text{ for }\lambda\in S \right\}.
\]
We refer to sets of $L^2[0,1]$ of this form as finite monomial spaces. The following definition appears in \cite{AglerMcCarthy2023_Beurling}.

\begin{definition}
    We say that  a subspace $\MM$ of $L^2[0,1]$ is a monomial space if there exists a sequence $\MM_n$ of finite monomial spaces such that $\MM_n \to \MM$ as $n \to \infty$, i.e.,
    \[
        \MM = \left\{f \in L^2[0,1] \, : \, \lim_{n \to \infty} \dist(f,\MM_n) = 0 \right\}.
    \]
\end{definition}

Equipped with these definitions we can state \cite[Theorem 1.4]{AglerMcCarthy2023_Beurling}, adding the characterization of the closed subspaces which are invariant under $T(t)$ for every $t>0$.

\begin{theorem}
    \label{th:characterization_lattice}
    Let $M$ be a closed non-zero subspace of $L^2[0,1]$. Then the following are equivalent:\begin{itemize}
        \item[(i)] $M \in \lat \CC$,
        \item[(ii)] $M$ is a monomial space,
        \item[(iii)] $M \in \cap_{t>0} \lat T(t)$.
    \end{itemize}
\end{theorem}
\begin{proof}
    The equivalence between items (i) and (ii) is established in \cite[Theorem 1.4]{AglerMcCarthy2023_Beurling}. The implication (iii) $\implies$ (i) follows from \eqref{eq:Lat_T(t)_subset_C}.

    To show (ii) $\implies$ (iii), let $M$ be a monomial space in $L^2[0,1]$ and $f \in M$. By definition, there exists a sequence of finite monomial spaces $\MM_n$ such that $\MM_n \to \MM$. Since $T(t)\MM_n \subseteq \MM_n$ (as $\MM_n$ is the span of eigenfunctions of $T(t)$), we have:
    \[
    \begin{aligned}
        \dist(T(t)f, \MM_n) & = \inf_{m \in \MM_n} \norm{T(t)f - m} \leq \inf_{m' \in T(t)\MM_n} \norm{T(t)f - m'} \\
        & = \inf_{m \in \MM_n} \norm{T(t)f - T(t)m} \leq \norm{T(t)} \inf_{m \in \MM_n} \norm{f - m} \to 0.
    \end{aligned}
    \]
    Thus, $T(t)f \in M$ for every $t > 0$, completing the proof.
\end{proof}

\begin{remark} It is known that for a contraction $C_0$-semigroup, the lattice of the closed invariant subspaces of the semigroup is the same that the lattice of the closed invariant subspaces of the cogenerator (see \cite[Chap. 2, Theorem 10.9]{Fuhrmann1981}). In particular, $(e^{t}T(2t))_{t\geq 0}$ is contraction $C_0$-semigroup on $L^2[0,1]$ generated by $2A+I,$ and whose cogenerator is $V=((2A+I)+I)((2A+I)-I)^{-1} = A^{-1}+I$,  see \Cref{pr:boundedness_semigroup} and \Cref{th:semigroups}. Therefore, by \Cref{rem5.5} one can write $$-\mathcal{C}= A^{-1} = V-I,$$ and the equivalence between items (i) and (iii) on \Cref{th:characterization_lattice} can be also deduced from this remark.

\end{remark}

To conclude this section, we present results on the minimal invariant subspaces of $T(t)$ for a fixed $t>0$. To know the minimal invariant subspaces of $T(t)$ would help to come near to the solution of the ISP. It is known that if $M$ is a minimal invariant subspace of $T(t)$, then 
\[
    M = K_f := \overline{\operatorname{span} \{ f, T(t)f, T(t)^2 f, \dots \}}^{L^2[0,1]}
\]
for some nonzero $f \in M$. Therefore, to study minimal invariant subspaces of $T(t)$, we need to understand the cyclic subspaces $K_f$ for $f \in L^2[0,1]$.

\begin{theorem}
    Let $\lambda \in \C$ such that $\Real \lambda < -1/2,$ $g = g_\lambda f$ for some $f \in L^2[0,1],$ and $L\in\C\setminus\{0\}$ such that $\lim_{s \to 0^+} f(s) = L$. Then $g_\lambda \in K_g$. In particular $K_g$ is minimal if and only if $f$ is a non-zero constant, that is, $K_g$ is one dimensional generated by the eigenfunction $g_{\lambda}.$
\end{theorem}
\begin{proof}
First note that  $g_\lambda f \in L^2[0,1]$ for $\Real \lambda < -1/2$ since $\lim_{s \to 0^+} f(s)$ is a complex number. 
    
Now consider the sequence $(e^{-nt\lambda} (T(t)^n g)_{n\in\N},$ which is clearly contained in $K_g$ and fix $\e>0$. Notice that,
    \[
        \begin{aligned}
            \int_0^1 \abs{e^{-nt\lambda} T(nt) g(s) - L g_\lambda(s)}^2 \, ds & = \int_0^1 \abs{e^{-nt(\lambda+1)} g(e^{-nt}s) - L g_\lambda(s)}^2 \, ds \\ 
            & = \int_0^{1} \abs{ g_\lambda(s)}^2 \abs{f(e^{-nt}s) - L}^2 \, ds.
        \end{aligned}
    \]
    Now, as $\lim_{s \to 0^+} f(s) = L \neq 0$, there exists $\delta>0$ such that if $u \in (0,\delta)$ then $\abs{f(u) - L} < \sqrt{-\e(2\Real \lambda +1)}$. Take $n$ big enough such that $e^{-nt} < \delta$, then
    \[
        \begin{aligned}
            \int_0^1 \abs{e^{-nt\lambda} T(nt) g(s) - L g_\lambda(s)}^2 \, ds & \leq -\e(2\Real \lambda +1)\int_0^{1} \abs{ g_\lambda(s)}^2 \, ds = \e.
        \end{aligned}
    \]
    So, $Lg_\lambda \in K_g,$ and $K_g$ is minimal if and only if $f$ is a non-zero constant because the subspace generated by $g_\lambda$ is one-dimensional.
\end{proof}

\begin{theorem}Let $\lambda \in \C$ such that $\Real \lambda < -1/2,$ $m\in\N,$ $g = g_\lambda\log^m f$ for some $f \in L^2[0,1],$ and $L\in\C\setminus\{0\}$ such that $\lim_{s \to 0^+} f(s) = L$. Then $g_\lambda \in K_g.$ In particular $K_g$ contains the eigenfunction $g_{\lambda},$ and therefore it is not a minimal closed invariant subspace.

\end{theorem}
\begin{proof}

First note that  $g_\lambda \log^m f  \in L^2[0,1]$ for $\Real \lambda < -1/2$ since $\lim_{s \to 0^+} f(s)$ is a complex number. 

Now consider the sequence $ (\frac{e^{-nt \lambda}T(nt)g}{(nt)^m})_{n\in\N},$ which is contained in $K_g$ and fix $\e>0$. Denote by $h_n = \frac{e^{-nt \lambda}T(nt)(g_{\lambda}\log^m)}{(nt)^m}.$ On one hand
    \[
        \int_0^1 \abs{\frac{e^{-nt \lambda}T(nt)g(s)}{(nt)^m} - L h_n(s)}^2  ds  = \frac{1}{(nt)^{2m}} \int_0^1 \abs{ g_{\lambda}(s)\log^m(e^{-nt}s)}^2 \abs{f(e^{-nt}s) - L}^2 ds.
    \]
    As $\lim_{s \to 0^+} f(s) = L \neq 0$, there exists $\delta>0$ such that if $s \in (0,\delta)$ then $\abs{f(s) - L} < \e$. Take $n$ big enough such that $e^{-nt} < \delta$, then
    \[
      \begin{aligned}  \int_0^1 \abs{\frac{e^{-nt \lambda}T(nt)g(s)}{(nt)^m} - L h_n(s)}^2 \, ds & \leq  \frac{\e^2}{(nt)^{2m}} \int_0^1 \abs{ g_{\lambda}(s)\log^m(e^{-nt}s)}^2 \, ds \\
      &=\e^2 \int_{0}^{1} \abs{g_{\lambda}(s)}^2\abs{\sum_{j=0}^m\frac{(-1)^{m-j}\log^j(s)}{(nt)^j}}^2\\
      &\leq  \e^2  \sum_{j=0}^m \frac{1}{t^j}\int_0^{1} \abs{g_{\lambda}(s)\log^j(s)}^2\,ds \\
      &= C  \e^2,
      \end{aligned}
    \] with $\displaystyle C=\sum_{j=0}^m \frac{1}{t^j}\int_0^{1} \abs{g_{\lambda}(s)\log^j(s)}^2\,ds.$
    
    On the other hand,  \[
      \begin{aligned}  \int_0^1 \abs{L h_n(s)-(-1)^m L g_{\lambda}(s)}^2 \, ds & = L^2 \int_0^1 \abs{ g_{\lambda}(s)}^2\abs{\frac{\log^m(e^{-nt}s)}{(nt)^m}-(-1)^m}^2 \, ds \\
      &= L^2 \int_0^1 \abs{g_{\lambda}(s)}^2\abs{\sum_{j=1}^m\frac{(-1)^{m-j}\log^j(s)}{(nt)^j}}^2 \,ds \\
         &\leq L^2 \sum_{j=1}^m \frac{1}{(nt)^j}\int_0^{1} \abs{g_{\lambda}(s)\log^j(s)}^2\,ds  \to 0,
      \end{aligned}
    \] as $n\to\infty.$ So we conclude that $Lg_\lambda \in K_g.$
\end{proof}

\printbibliography 

\end{document}